\documentclass[11pt]{article}
\usepackage{url,hyperref,comment}
\usepackage{graphicx,amsfonts,amsmath,amsthm,amssymb}
\usepackage{color}
\usepackage[ruled]{algorithm2e}

\definecolor{darkgreen}{RGB}{0,160,0}

\newcommand{\C}{\mathcal{C}}
\newcommand{\R}{\mathbb{R}}
\newcommand{\E}{\mathbb{E}}
\newcommand{\N}{\mathbb{N}}

\newcommand{\1}{\mathbf{1}}

\newcommand{\Var}{\operatorname{Var}}

\renewcommand{\hat}{\widehat}
\renewcommand{\tilde}{\widetilde}
\renewcommand{\d}{\ensuremath {\,\mathrm{d}}}
\renewcommand{\vec}[1]{{\boldsymbol{#1}}}
\renewcommand{\P}{\mathbb{P}}

\newtheorem{lemma}{Lemma}

\newtheorem{proposition}{Proposition}

\newtheorem{remark}{Remark}

\title{{\bf Adaptive Compressed Sensing for Support Recovery of Structured Sparse Sets}}
\author{Rui~M.~Castro and Ervin~T\'{a}nczos%
\thanks{The authors are with the Department of Mathematics, Eindhoven University of Technology, 5600 MB Eindhoven,
The Netherlands (email \texttt{\url{e.t.tanczos@tue.nl}} and \texttt{\url{rmcastro@tue.nl}}).  This work was partially supported by NWO Grant 613.001.114.}}

\allowdisplaybreaks[1]
\begin{document}

\maketitle

\begin{abstract}
This paper investigates the problem of recovering the support of structured signals via adaptive compressive sensing.  We examine several classes of structured support sets, and characterize the fundamental limits of accurately recovering such sets through compressive measurements, while simultaneously providing adaptive support recovery protocols that perform near optimally for these classes.  We show that by adaptively designing the sensing matrix we can attain significant performance gains over non-adaptive protocols.  These gains arise from the fact that adaptive sensing can: (i) better mitigate the effects of noise, and (ii) better capitalize on the structure of the support sets.
\end{abstract}


\section{Introduction} \label{sec:intro}

Compressive sensing provides an efficient way to estimate signals that have a sparse representation in some basis or frame \cite{CS_candes_2006}, \cite{CS_Donoho_2006}, \cite{Dantzig_Candes_2007}, \cite{IntroCS_Candes_2008}, \cite{Lasso_Wainwright_2009}.  If the measurements can be chosen in a sequential and adaptive fashion it is possible to achieve further performance gains in the sense that weaker signals can be estimated more accurately than in the non-adaptive setting \cite{AS_Rui_2012},\cite{ACS_Malloy_2012}.  Furthermore, in some situations the signal may have additional structure that can be exploited.  For instance, in gene expression studies the signals of interest are supported on a submatrix of the gene-expression matrix, and are not arbitrary sparse signals.  In network monitoring anomalous behavior may ``radiate'' from infected nodes creating star-shaped patterns in the network graph.  The natural question that arises is if further performance gains can be realized using this structural information when estimating signals using compressive measurements? Furthermore, can adaptively and sequentially designing the sensing actions provide further performance gains over non-adaptive schemes? The answer to both questions is essentially affirmative, and this work quantifies such gains in a general way.

\textbf{Contributions.} In this work we investigate the problem of recovering the support of structured sparse signals using adaptive compressive measurements.  Our focus is on the performance gains one can achieve when adaptively designing the sensing matrix compared to the situation where the sensing matrix is constructed non-adaptively.  Furthermore, our aim is to highlight the way in which adaptive compressed sensing can capitalize on structural information.  An appealing feature of compressed sensing is that accurate estimation can be done using only a few measurements.  With this in mind we design algorithms for this problem that are sample-efficient, in the sense that they collect a number of observations that is not larger than the sample complexity of the best non-adaptive strategies.

The classes of structured support sets under consideration in this paper are
\begin{itemize}
\item{\textbf{$s$-sets:} any subset of $\{1,\ldots,n\}$ with size $s$}
\item{\textbf{$s$-intervals:} sets consisting of $s$ consecutive elements of $\{1,\ldots,n\}$}
\item{\textbf{unions of $s$-intervals:} unions of $k$ disjoint $s$-intervals}
\item{\textbf{$s$-stars:} any star of size $s$ in a complete graph (where the edges of the graph are identified with $\{1,\ldots,n\}$)}
\item{\textbf{unions of $s$-stars:} unions of $k$ disjoint $s$-stars}
\item{\textbf{$s$-submatrices:} any submatrix of a given size $s_r \times s_c$ of an $n_r \times n_c$ matrix}
\end{itemize}

We analyze the fundamental limits of recovering support sets for the above classes under non-adaptive and adaptive sensing paradigms.  This is done by showing performance lower bounds for both adaptive and non-adaptive sensing.  We also provide adaptive sensing protocols with near optimal performance to show the tightness of the lower bounds, and to illustrate how adaptive compressed sensing can capitalize on the structure of the support sets in the estimation.  Finally, we provide procedures that next to being near optimal in a statistical sense also perform estimation using only a small number of measurements and are thus appealing from a practical point of view.

Note that, while adaptive compressive measurements might be very advantageous from a statistical and computational point of view, they also require a flexible infrastructure and hardware.  In some settings, like that of the \emph{single-pixel} camera \cite{Single_pixel_2008}, all the necessary infrastructure is already in place.  In tomography and magnetic resonance imaging the use of adaptive compressive samples is also possible, as described in \cite{Deutsch_2009, Panych_1994}.  It is important to note that in the latter settings one has additional physical constraints that need to be accounted for.  Other motivating examples include applications in sensor networks and monitoring, for instance identifying viruses in human or computer networks, or gene-expression studies, for instance when we have a group of genes co-expressed under the influence of a drug, or we have patients exhibiting similar symptoms \cite{Yoon_2005, Moore_2010}.  The results in this paper are foundational in nature, and aim at understanding the draws and limitations of adaptive compressive sensing in the context of structured support recovery.  Furthermore, our model fits the case where ``compression'' happens in the physical domain and before sensing takes place (this fits settings in \cite{Single_pixel_2008,Deutsch_2009, Panych_1994}).  It is important to note that if the sensing is further constrained (so that the measurement vectors cannot be arbitrary) then the performance of any algorithm will be affected.  For a discussion on how such constraints can effect the performance of adaptive compressive sensing see e.g.  \cite{Constrained_Davenport_2015}.

\begin{table}[h]
\caption{Summary of scaling laws for the signal magnitude.}
\label{table}
\smaller
\begin{center}
\begin{tabular}{l|c|c|c|}
& Non-Adaptive Sensing & \multicolumn{2}{c|}{Adaptive Sensing}\\
& (necessary) & (necessary) & (sufficient)\\
\hline\hline
$s$-sets & $\sqrt{\frac{n}{m} \log n}$ & $\Big.  \sqrt{\frac{n}{m}\log s}$ & $\sqrt{\frac{n}{m}\log s}$ \\
\hline
unions of $k$ disjoint $s$-intervals & $\Big.  \frac{1}{s} \sqrt{\frac{n}{m}\log \frac{n}{ks}}$ & $\frac{1}{s} \sqrt{\frac{n}{m}\log ks}$ & $\frac{1}{s} \sqrt{\frac{n}{m}\log ks}$ \\
\hline
unions of $k$ disjoint $s$-stars & $\Big.  \sqrt{\frac{n}{m}\log \frac{\sqrt{n}}{ks}}$ & $\frac{1}{s} \sqrt{\frac{n}{m}\log ks}$ & $\Big.  \frac{1}{s} \sqrt{\frac{n}{m}\log ks}$ \\
\hline
$\sqrt{s} \times \sqrt{s}$ submatrices of & & &\\
an $\sqrt{n}\times\sqrt{n}$ matrix & $\Big.  \sqrt{\frac{n}{\sqrt{s}m}\log \frac{n}{s}}$ & $\frac{1}{s} \sqrt{\frac{n}{m}\log s}$ & $\frac{1}{s^{3/4}} \sqrt{\frac{n}{m}\log s}$ \\
\end{tabular}
\end{center}
Scaling laws for the signal magnitude $\mu$ (constants omitted) which are necessary/sufficient for $\max_{S\in\C} \E(\hat S\triangle S)\to 0$ as $n\to\infty$, where $\C$ denotes the corresponding class of support sets.  The results in the last column make some sparsity assumptions, meaning $s\ll n$.  For exact conditions see relevant propositions of Section~\ref{sec:sigstrength_proc}.  The results presented for $s$-sets are known (see for instance \cite{Fundamental_Limits_AS_Candes_2013, ACS_Malloy_2012, AS_Rui_2012}) and are presented for comparison purposes.  Furthermore, considering $k$ disjoint $s$-intervals, the sufficient condition can be derived using the algorithm of \cite{Singh_Graph_ACS_2013}, and the necessary condition for adaptive sensing was already derived in \cite{Balakrishnan_2012} for the case $k=1$.  Finally, the necessary and sufficient conditions for adaptive sensing in the case of submatrices do not match, and the necessary condition stated above is the one derived in \cite{Balakrishnan_2012}.
\end{table}
\vspace{10pt}

Table~\ref{table} summarizes some of our results, showing necessary and sufficient conditions for the signal magnitude for accurate support estimation.  It also highlights two different facets of the gains of adaptive sensing over non-adaptive sensing.  First, note that the necessary conditions of non-adaptive sensing include a $\sqrt{\log n}$ factor for each of the classes under consideration.  This factor is replaced by the logarithm of the sparsity when considering adaptive sensing, and this is due to the fact that adaptive strategies are better able to mitigate the effects of noise.  Second, for certain classes adaptive sensing can gain greater leverage from the structure of the support sets compared to non-adaptive sensing.  This phenomenon is best visible considering the class of $s$-stars, where estimators using non-adaptive sensing gain practically nothing from the structural information whereas adaptive sensing benefits greatly from it.  Note that the necessary and sufficient conditions for the class of submatrices using adaptive sensing do not match, and a full characterization of the problem in that case remains open.  We also remark at this point that the results derived in this paper are non-asymptotic in nature and also account for the constant factors in the scaling laws.  The asymptotic presentation in Table~\ref{table} merely makes it easier to highlight the main contributions of the work.

\textbf{Related work.} The current work is built on a number of recent contributions on detection and estimation of sparse signals using compressive sensing.  Considering general sparse signals without structure \cite{Fundamental_Limits_AS_Candes_2013} and \cite{AS_Rui_2012} provide theoretical performance limits of adaptive compressive sensing, characterizing the gains one can realize when adaptively designing the sensing matrix.  Complementing these results, \cite{DistilledCS_Haupt_2009,BinSearch_Malloy_2012} and \cite{ACS_Malloy_2012} provide efficient near optimal procedures for estimation.  Considering the problem of detection \cite{CSDetection_AriasCastro_2012} provides both theoretical limits and optimal procedures both in the non-adaptive and adaptive compressed sensing settings.

The problem of estimating structured sparse signals was examined in the past in a multitude of different settings.  In the normal means model several graph structures were considered in \cite{Trail_Arias_2008,Comb_Testing_Lugosi_2010,Ery_ClusterDet_2011,Ery_CommunityDet_2013,Saligrama_2014,Sharpnack_Changepoint_2012}, such as connected components on a lattice, sub-graphs with a small cut size, tree structures and so on.  In \cite{Butucea_2011} the authors consider estimating a submatrix of a high dimensional matrix in the non-adaptive framework.  All the previously cited work share in common that coordinate-wise observations are considered and that the observations are collected in a non-adaptive manner.

The authors of \cite{AS_structured_Rui_ET_2013} also consider the problem of recovering structured supports using coordinate-wise observations, but in a setting where these are collected in a sequential and adaptive manner.  Therefore this paper can be seen as an extension of \cite{AS_structured_Rui_ET_2013} in that we move from coordinate-wise observations to compressed sensing.  Although some of the techniques and insights can be used from that work, changing the measurement model introduces a number of new challenges to tackle.  In particular the information provided by compressive measurements is very different in nature from that provided by coordinate-wise observations.  This means that structural information is captured in the observations in a different way, which influences both the theoretical limits and the way support recovery procedures need to be designed.  For a more extensive literature review in the setting of coordinate-wise observations we also direct the reader to \cite{AS_structured_Rui_ET_2013}.

Structured support recovery problems have been investigated in the compressive sensing setting as well.  In \cite{ModelBasedCS_Baraniuk_2010} the authors consider recovering tree-structured signals in the non-adaptive framework and show that using structural information enhances the performance of compressive sensing methods.  Recovering tree-structured signals is also the topic of \cite{TreeStructure_Soni_2011} and \cite{Haupt_TreeSparse_2014} but in these works the problem is examined in the adaptive sensing setting.  In these works the authors consider signals in which the activation pattern is a rooted subtree of a given tree and show that one can realize further gains recovering these types of supports by adaptively designing the sensing matrix.  Our work is closely related, but the structured class investigated in \cite{TreeStructure_Soni_2011, Haupt_TreeSparse_2014} is clearly different from the ones listed in Table~\ref{table}.  The work in \cite{Singh_Graph_ACS_2013} considers activation patterns that have low cut-size in an arbitrary (fixed) signal graph and also find that adaptivity enhances the statistical performance of compressive sensing.  Though these types of classes seem more close to the ones investigated in this paper, note that most classes in Table~\ref{table} do not result in a lower cut-size then an arbitrary $s$-sparse set, meaning that these can not be efficiently encoded with the definitions of \cite{Singh_Graph_ACS_2013}.  As an example, arbitrary submatrices in a 2d-lattice have typical cut-size on the order of $s$, the same as any $s$-sparse subset of the 2d-lattice.  Similar comments apply to the other classes considered in this paper as well.  Furthermore in our work we provide much sharper lower bounds than those in \cite{Singh_Graph_ACS_2013}, as we explicitly take into consideration the structural properties of the signal classes. Moving away from graphs, \cite{Balakrishnan_2012} investigates the problem of finding block-structured activations in a signal matrix considering both non-adaptive and adaptive measurements.  \cite{Balakrishnan_2012} reports similar findings to the previous authors, namely that both adaptivity and structural information provide gains in support recovery when dealing with block-activations in a matrix.  Our work extends these results by investigating general structured activations.  Finally, the sample complexity of compressive sensing was studied in \cite{Samplecomp_Saligrama_2013} and \cite{Samplecomp_Saligrama_2014} for the support recovery of general sparse signals in the non-adaptive and adaptive sensing settings respectively.

It is instructive to note a fundamental difference between non-adaptive sensing and adaptive sensing problems.  In non-adaptive sensing support recovery methods can often be computationally demanding or even intractable, a prominent example being submatrix estimation \cite{Butucea_2011, Balakrishnan_Tradeoffs_2011, Complexity_lower_Berthet_2013}.  Contrasting this, adaptive sensing algorithms can solve this problem using polynomial-time algorithms.  Though this might seem surprising, one has to bear in mind that there is a fundamental difference between the two setups.  In fact when using adaptive sensing one already shakes most of the computational burdens by tailoring the sample to facilitate inference.  The bottleneck of such algorithms lies in sample collection, but given a good strategy the sample will contain much less confounders making the inference itself easier computationally.  This, next to increased statistical power, can be another appealing reason for using adaptive sensing methods whenever possible.

The paper is structured as follows.  Section~\ref{sec:setting} describes the problem setting in detail.  In Section~\ref{sec:sigstrength} we provide adaptive sensing procedures for structured support recovery and analyze the theoretical limits of the problem, both under non-adaptive and adaptive sensing paradigms.  In this section we only make a restriction to the sensing power available, but not on the number of projective measurements we are allowed to make.  In Section~\ref{sec:sample} we further restrict the number of measurements.  Finally we provide some concluding remarks in Section~\ref{sec:conc}.



\section{Problem Setting} \label{sec:setting}

In this work we consider the following statistical model.  Let $\vec{x}=(x_1,\ldots,x_n)\in\R^n$ be a vector of the form
\begin{equation}\label{eqn:signal_model}
x_i= \left\{ \begin{array}{ll}
\mu & \textrm{, if $i \in S$}\\
0 & \textrm{, if $i \notin S$}
\end{array}\right.  \ ,
\end{equation}
where $\mu>0$ and $S$ is an unknown element of a class of sets denoted by $\C$.  We refer to $\vec{x}$ as the \emph{signal} and to $S$ as the \emph{support} or \emph{significant/active components} of the signal.  The set $S$ is our main object of interest.  The signal model \eqref{eqn:signal_model} may seem overly restrictive at first because of the fact that each non-zero entry has the same value $\mu$.  However, our lower bounds and the procedures of Section~\ref{sec:sample_proc} can be generalized to signals with active components of arbitrary magnitudes and signs, in which case the value $\mu$ would play the role of the minimal absolute value of the non-zero components.  For sake of simplicity we do not discuss this extension here, but refer the reader to \cite{Comb_Testing_Lugosi_2010}, \cite{CSDetection_AriasCastro_2012}, \cite{ACS_Malloy_2012} for details on how this can be done.

We are allowed to collect multiple measurements of the form
\begin{equation} \label{eqn:measurement}
Y_j = <A_j ,\vec{x}> + W_{j}, \  j=1,2,\ldots\ ,
\end{equation}
where $j$ indexes the $j$th measurement.  Thus each measurement is the inner product of the signal $\vec{x}$ with the vector $A_j \in \mathbb{R}^n$, contaminated by Gaussian noise.  The noise terms $W_{j} \sim \mathcal{N}(0,1)$ are independent and identically distributed (i.i.d.) standard normal random variables, also independent of $\{ A_i \}_{i=1}^{j}$.  Under the adaptive sensing paradigm $A_j$ are allowed to be functions of the past observations $\{Y_i ,A_i \}_{i=1}^{j-1}$.  This model is only interesting if one poses some constraint on the total amount of sensing energy available.  Let $A$ denote the matrix whose $j$th row is $A_j$.  We require
\begin{equation}\label{eqn:budget}
\sup_{S \in \C} \ \E_S \left( \| A \|^2_F \right)\ =\ \E_S\left(\sum_j \|A_j\|^2 \right)\ \leq m\ ,
\end{equation}
where $\|\cdot \|_F$ denotes the Frobenius norm, $m$ is our total energy budget, and $\E_S$ denotes the expectation with respect to the joint distribution of $\{ A_j ,Y_j \}_{j=1,2,\dots}$ when $S \in \C$ is the support set.  It is also possible to consider algorithms satisfying an exact energy constraint as opposed to the expected energy constraint in \eqref{eqn:budget}.  This requires an extension of the arguments in the body of the paper, but yields essentially the same results.  For the sake of completeness we provide such details in Appendix~\ref{app:no_expectations}.


\subsection{Inference Goals} \label{sec:goals}

This work aims at characterizing the difficulty of recovering structured sparse signal supports with adaptive compressive sensing.  We are interested in settings where the class $\C$ contains sets with some sort of structure, for instance the active components of $\vec{x}$ are consecutive.  For the unstructured case, that is, when $\C$ contains every set of a given cardinality, there already exists a lower bound in \cite{AS_Rui_2012}, and a procedure that achieves this lower bound in \cite{ACS_Malloy_2012}.  The main goal of this work is to provide results for the problem of recovering structured sparse sets.

We are interested in two aspects of adaptive compressive sensing.  First, given $n$, $m$, $\varepsilon$ and $\C$ the aim is to characterize the minimal signal strength $\mu$ for which $S$ can be reliably estimated, which means there is an algorithm and sensing strategy such that for a given $\varepsilon>0$, 
\begin{equation} \label{eqn:goal}
\max_{S \in \C} \E_S ( | \hat{S} \triangle S|) \leq \varepsilon \ ,
\end{equation}
where $\hat{S} \triangle S$ is the symmetric set difference.  Furthermore, we aim to construct such an adaptive sensing strategy.  Although the setting above makes sense whenever $\varepsilon \in [0,|S|]$, the problem is only interesting when $\varepsilon$ is small.  Hence we will take $\varepsilon$ as an element of $[0,1]$.  We remark at this point that our main interest lies in the scaling of $\mu$ in terms of the model parameters, but we do not aim to find accurate constants.  With this in mind, the procedures throughout the paper could be improved with more careful and refined analysis.  However, these improvements would only improve constant factors, and so we chose to keep technicalities to a minimum providing a smoother presentation at the price of suboptimal constants.

\begin{remark}
As mentioned in the introduction, this work can be seen as an extension of \cite{AS_structured_Rui_ET_2013} from component-wise sampling to the more general compressive sensing, and it is instructive to briefly discuss the differences between the two setups.  Component-wise observations can be viewed as restricting compressive sensing by requiring each measurement vector $A_i$ to have exactly one non-zero entry (though the problem is set up a bit differently in \cite{AS_structured_Rui_ET_2013} the two are effectively the same).  It is shown in \cite{AS_structured_Rui_ET_2013} that the necessary conditions for support recovery for the classes considered in Table~\ref{table} are as follows: the condition is the same for $s$-sets, while for the other classes the $\tfrac{1}{s}$ term moves inside the square-root if one only allows component-wise observations.  Also, these conditions are sufficient in the case of component-wise samples.

Lying at the heart of the difference between the rates for support recovery between the two setups is the increased detection power of compressive sensing over coordinate-wise sampling.  In a nutshell, detection of a signal is the problem of differentiating two hypotheses: the null being that all signal components are zero and the alternative being that there are $s$ non-zero components somewhere in the signal vector.  \cite{CSDetection_AriasCastro_2012} shows that the necessary and sufficient conditions for detection for compressive sensing is $\tfrac{1}{s}\sqrt{\tfrac{n}{m}}$, whereas \cite{AS_Rui_2012} shows the same for component-wise sampling to be $\sqrt{\tfrac{1}{s}\tfrac{n}{m}}$.  When moving from component-wise sampling to compressive sensing, for certain structured classes it is possible to make use of this increased detection power, which in turn lowers the requirement for the signal magnitude.  This also means algorithms need to be designed with a different mindset when using compressive sensing instead of coordinate-wise sampling.
\end{remark}

Second, given $n$, $m$, $\mu$, $\varepsilon$ and $\C$ we wish to characterize the minimal number of samples needed to ensure \eqref{eqn:goal}.  Considering the unstructured case, we know that non-adaptive procedures need at least $O( s \log \frac{n}{s})$ measurements \cite{Samplecomp_Saligrama_2013} and that this bound is achievable \cite{IntroCS_Candes_2008} (these results apply when the signal strength $\mu$ is close to the threshold of estimability).  On the other hand, to the best knowledge of the authors, an exact characterization of the sample complexity for adaptive procedures is not yet available, though there has been work done on the topic \cite{Samplecomp_Saligrama_2014}.  In that work the authors present a result that states the sample complexity of the problem scales essentially as $s$.  However, it is not clear if that bound is tight.  In Section~\ref{sec:conc} we provide more insight on this question.

In what follows we use the symbol $\1$ to denote both the usual indicator function (e.g., $\1\{i\in S\}$ takes the value 1 if $i\in S$ and zero otherwise), and to denote binary vectors with support $S$.  For instance $\1_S$ denotes an element of $\{0,1\}^n$ for which the entries in $S$ have value 1 and all the other entries have value 0.  Note that to ease distinction of the two the arguments of the functions are in a different place (after the symbol in the first case and in the subscript of the symbol in the second case).  Furthermore, let $\P_S$ denote the joint distribution of $\{ A_j ,Y_j \}_{1,2,\dots}$ when $S \in \C$ is the support set, and $\E_S$ denote the expectation with respect to $\P_S$.

\vspace{0.1cm}



\section{Signal strength} \label{sec:sigstrength}

We now examine the minimal signal strength required to recover structured support sets.  In this setup we are allowed to make a potentially infinite amount of measurements of the form \eqref{eqn:measurement} (provided the budget \eqref{eqn:budget} is satisfied).  Although this might not be reasonable from a practical standpoint, it is a good place to start understanding the fundamental performance limits of adaptive compressing sensing, and we will see in Section~\ref{sec:sample} that the same performance can be attained with a small number of measurements.


\subsection{Procedures} \label{sec:sigstrength_proc}

It is instructive to briefly consider a simple support recovery algorithm for the unstructured case.  When the support set can be any set of a given cardinality and there is no restriction on the number of samples we are allowed to take the situation becomes similar to that of \cite{AS_structured_Rui_ET_2013}, where the authors consider coordinate-wise observations.  A simple procedure in this case is to perform a Sequential Likelihood Ratio Test (SLRT\footnote{In the literature this sequential procedure is also referred to as the Sequential Probability Ratio Test (SPRT) (see e.g., \cite{SLRT_Wald_1945}).  We feel, however, that the use of the term ``likelihood ratio'' is perhaps more appropriate, as in most settings one is computing a ratio between densities and not probabilities.}) for each coordinate separately.  More precisely for every coordinate $i =1,\dots ,n$ collect observations of the form
\[
Y_{i,j} = a x_i + W_j = <a \1_{\{i\}} , \vec{x}> + W_j \ , j=1,\dots ,N_i \ ,
\]
with some fixed $a>0$, where we recall that $\1_{\{i\}}$ is a singleton vector.  The number of observations $N_i$ is random and is given by
\[
N_i = \min \big\{ n \in \N : \ \sum_{j=1}^{n} \log \frac{\d \P_1 (Y_{i,j})}{\d \P_0 (Y_{i,j})} \notin (l,u) \big\} \ ,
\]
where $\P_0$ ($\P_1$) is the distribution of the observations when component $i$ is non-active (active), and $l<0<u$ are the lower and upper stopping boundaries of the SLRT.  Then our estimator $\hat{S}$ will be the collection of components $i$ for which the log-likelihood process above hits the upper stopping boundary $u$.  Considering the test of component $\vec{x}_i$ we have the following.

\begin{lemma} \label{lemma:SLRT}
Set $l = \log \frac{\beta}{1-\alpha}$ and $u=\log \frac{1-\beta}{\alpha}$ with $\alpha ,\beta \in (0,1/2)$, and let the type I and type II error probabilities of the SLRT described above be $\alpha_a$ and $\beta_a$.  Then $\alpha_a \to \alpha$ and $\beta_a \to \beta$ as $a \to 0$.  Furthermore
\[
a^2 \E_0 (N_i) \leq \frac{2}{\mu^2} \left( \alpha \log \frac{\alpha}{1-\beta} + (1-\alpha ) \log \frac{1-\alpha}{\beta} \right) \leq \frac{2}{\mu^2} \log \frac{1}{\beta}
\]
and
\[
a^2 \E_1 (N_i) \leq \frac{2}{\mu^2} \left( \beta \log \frac{\beta}{1-\alpha} + (1-\beta ) \log \frac{1-\beta}{\alpha} \right) \leq \frac{2}{\mu^2} \log \frac{1}{\alpha}
\]
as $a \to 0$.
\end{lemma}
\begin{proof}
The proof goes the same way as that of Proposition~1 in \cite{AS_structured_Rui_ET_2013}.
\end{proof}

Using the previous result we can immediately analyze the procedure above.  Set $\alpha = \varepsilon /2n$ and $\beta = \varepsilon /2s$ in the proposition above, and choose $a$ to be arbitrarily small.  Hence $\alpha_a$ and $\beta_a$ will be close to the nominal error probabilities $\alpha$ and $\beta$ and we ensure \eqref{eqn:goal}.  Then using the other part of Lemma~\ref{lemma:SLRT} we can upper bound the expected energy used by the tests.  Summing this over all the tests and using \eqref{eqn:budget} we arrive at the following.

\begin{proposition}\label{prop:vanilla1}
Testing each component $\vec{x}_i , \ i=1,\dots ,n$ as described above yields an estimator satisfying \eqref{eqn:budget} and \eqref{eqn:goal} whenever
\[
\mu \geq \sqrt{\frac{2n}{m} \log \frac{2s}{\varepsilon} + \frac{2s}{m} \log \frac{2n}{\varepsilon}} \ .
\]
\end{proposition}

When the support is sparse, the first term dominates the bound above.  This coincides with the lower bound of \cite{AS_Rui_2012} showing that the simple procedure above is near optimal.

\begin{remark}\label{remark:sparsity1}
Note that the lower bound presented in \cite{AS_Rui_2012} is valid for a slightly broader class than the $s$-sets, namely one also has to include $(s-1)-sets$ into the class.  However, the procedure outlined above works without any modifications for this broadened class as well, and so the result of Proposition~\ref{prop:vanilla1} holds for this larger class.  A similar comment applies to all the procedures presented later on: the procedures are presented for classes of a given sparsity for sake of clarity, but the analysis shows that they also work for classes containing sets of slightly different sparsity.  This is important to note as because of technical reasons the some of the lower bounds of Section~\ref{sec:sigstrength_lower} can only deal with such enlarged classes.
\end{remark}

The procedures for recovering structured support sets will be very similar in nature, but slightly modified to take advantage of the structural information.  In particular we know from \cite{CSDetection_AriasCastro_2012} that it is possible to detect the presence of weak signals using compressive sensing.  In order to take advantage of this property our procedures consist of two phases: a \emph{search phase} and a \emph{refinement phase}.  The aim of the search phase is to find the approximate location of the signal using a detection type method, that is identifying a subset of components $\vec{P} \subset \{ 1,\dots ,n\}$ such that $|\vec{P}| \ll n$ and $S \subset \vec{P}$ with high probability.  Once this is done we can focus our attention exclusively on $\vec{P}$ in the refinement phase and estimate the support in the same manner as in the unstructured case. This general approach is similar in spirit to that of \cite{Singh_Graph_ACS_2013}.


\subsubsection{Unions of $s$-intervals}

The first structured class we consider is the unions of $k$ disjoint $s$-intervals.  Note that with $k=1$ this is a special case of the class considered in \cite{Balakrishnan_2012} when the signal matrix has one row.  The unions of intervals class is a good starting point to highlight the main ideas of how recovery algorithms can benefit from structural information in the adaptive compressed sensing setting, particularly because it can be viewed as a bridge between the unstructured case (with $k=s$ and intervals of length one) to the most structured class ($k=1$).  It is worth noting that, by using an appropriate instantiation of the algorithm in \cite{Singh_Graph_ACS_2013} one can get similar performance guarantees to the ones presented here.  In particular one needs to consider a line graph and use the dendogram construction described in Section~2.3 of \cite{Singh_Graph_ACS_2013}.

Consider the class of sets that are unions of $k$ disjoint intervals of length $s$.  Formally,
\[
\C = \big\{ S \subset \{ 1,\dots ,n\} : \ S=\bigcup_{i=1}^k S_i \ , \ S_i = \{ l_i,\dots ,l_i +s-1 \} , \ S_i \cap S_j = \emptyset \ \forall i \neq j \big\} \ .
\]

In principle we can also consider overlapping intervals (not enforcing these are disjoint).  Although this can still be handled in a similar fashion as done below it would result in a more cluttered presentation.

Our procedure for estimating $S$ is as follows.  Split the index set $\{ 1,\dots ,n \}$ into consecutive bins of length $s/2$ denoted by $\vec{P}^{(1)}, \dots ,\vec{P}^{(2n/s)}$.  We suppose $2n$ is divisible by $s$, as it makes the presentation less cluttered.  The procedure can be easily modified this is not satisfied.  Of these bins at least $k$ (and at most $2k$) are contained entirely in $S$.  In the search phase we aim to find the approximate location of the support by finding $k$ such bins.  To do this we test the following hypotheses
\[
H_0^{(i)}: \ \vec{P}^{(i)} \cap S = \emptyset \qquad \textrm{versus} \qquad H_1^{(i)}: \ \vec{P}^{(i)} \subset S \qquad i=1,\dots ,2n/s \ .
\]
We use a SLRT to decide between $H_0^{(i)}$ and $H_1^{(i)}$ for each $i=1,\dots ,n$, all with the same type I and type II error probabilities $\alpha$ and $\beta$.  The choices of $\alpha$ and $\beta$ and the exact way of carrying out the tests will be described later.  As an output of the search phase, we define the set $\vec{P}$ based on the tests above.  Since some $\vec{P}^{(i)}$ may only partially intersect the support $S$ we set $\vec{P}$ to be the union of those bins $\vec{P}^{(i)}$ for which either $H_1^{(i-1)}, H_1^{(i)}$ or $H_1^{(i+1)}$ was accepted.  This way we ensure $\P_S (S \nsubseteq \vec{P}) \leq 2k \beta$.  We also wish to ensure that $\vec{P}$ is small, and to do so we must to choose $\alpha$ appropriately.  Once this is done we can move on to the search phase and find the support within $\vec{P}$.  We can do this in a very crude way and use a similar procedure as in the unstructured case with type I and II error probabilities $\alpha' ,\beta'$.  The sensing energy used in this phase will be negligible due to $\vec{P}$ being small.  Finally the estimator $\hat{S}$ will be the collection of components that were deemed active at the end of the refinement phase.

We now choose $\alpha ,\beta ,\alpha' ,\beta'$ to ensure the estimator satisfies \eqref{eqn:goal}.  We have
\begin{align*}
\E_S \left( |\hat{S} \triangle S| \right) & \leq \E_S \left( |\hat{S} \triangle S| \ \big| S \nsubseteq \vec{P} \right) \P_S (S \nsubseteq \vec{P}) + \E_S \left( |\hat{S} \triangle S| \ \big| S \subseteq \vec{P} \right) \\
& \leq \E_S \left( \left.  |S \setminus \vec{P}| + \sum_{i \in \vec{P}: \ i \notin S} \alpha' + \sum_{i \in \vec{P}: \ i \in S} \beta' \ \right| S \nsubseteq \vec{P} \right) 2k \beta \\
& + n \alpha' + ks \beta' \ .
\end{align*}
Hence choosing $\alpha' = \varepsilon /4n ,\beta' = \varepsilon /4ks$ and $\beta = \varepsilon /8k^2 s^2$ ensures \eqref{eqn:goal}.  Note that $\alpha$ does not influence the probability of error.  However, it will influence the size of $\vec{P}$, and hence the total sensing energy required by the procedure.

To perform the $i$th test of the search phase we collect measurements using projection vectors of the form $a \1_{\vec{P}^{(i)}}$ with an arbitrarily small $a$ and perform a SLRT with stopping boundaries $l<0<u$.  Let $\E_0$ and $\E_1$ denote the expectation when $H_0^{(i)}$ or $H_1^{(i)}$ is true respectively.  Similarly to the unstructured case we now have the following.

\begin{lemma} \label{lemma:SLRT_int}
Set $l = \log \frac{\beta}{1-\alpha}$ and $u=\log \frac{1-\beta}{\alpha}$ with $\alpha ,\beta \in (0,1/2)$, and let the type I and type II error probabilities of the SLRT described above be $\alpha_a$ and $\beta_a$.  Then $\alpha_a \to \alpha$ and $\beta_a \to \beta$ as $a \to 0$.  Furthermore
\[
a^2 \E_0 (N_i) \leq \frac{2}{(s/2)^2 \mu^2} \left( \alpha \log \frac{\alpha}{1-\beta} + (1-\alpha ) \log \frac{1-\alpha}{\beta} \right) \leq \frac{2}{(s/2)^2 \mu^2} \log \frac{1}{\beta}
\]
and
\[
a^2 \E_1 (N_i) \leq \frac{2}{(s/2)^2 \mu^2} \left( \beta \log \frac{\beta}{1-\alpha} + (1-\beta ) \log \frac{1-\beta}{\alpha} \right) \leq \frac{2}{(s/2)^2 \mu^2} \log \frac{1}{\alpha}
\]
as $a \to 0$.
\end{lemma}

Using this we can upper bound the amount of sensing energy used for the test of $\vec{P}^{(i)}$ under $H_0^{(i)}$ and $H_1^{(i)}$.  However, now it is possible that neither statement in $H_0^{(i)}$ nor $H_1^{(i)}$ holds for a given bin $\vec{P}^{(i)}$.  Considering a test where neither of them is true we can still carry out the the same calculations as in Lemma~\ref{lemma:SLRT} and thus upper bound the expected sensing energy used for the test.

\begin{lemma} \label{lemma:bad_SLRT}
Set $l = \log \frac{\beta}{1-\alpha}$ and $u=\log \frac{1-\beta}{\alpha}$ with $\alpha ,\beta \in (0,1/2)$, and let $\tilde{s}$ denote the true number of signal components in $\vec{P}^{(i)}$.  Suppose that in the setting above neither $H_0^{(i)}$ nor $H_1^{(i)}$ is true, that is $0< \tilde{s} < s/2$.  Furthermore suppose $\tilde{s} \neq s/4$.  Then as $a \to 0$ we have
\[
a^2 \E_{\tilde{s}} (N_i) \leq \frac{2}{s \mu^2} \log \max \left\{ \frac{1-\alpha}{\beta} , \frac{1- \beta}{\alpha} \right\} \leq \frac{2}{s \mu^2} \log \frac{1}{\min \{ \alpha ,\beta \} } \ ,
\]
where $\E_{\tilde{s}}$ denotes the expectation when the number of signal components in $\vec{P}^{(i)}$ is $\tilde{s}$.
\end{lemma}
\begin{proof}
In what follows we drop the subscript $i$ to ease notation.  The log-likelihood ratio for an observation $Y_j$ is
\[
z_j = \log \frac{\d \P_1 (Y_j)}{\d \P_0 (Y_j)} = \frac{a s \mu Y_j}{2} - \frac{a^2 s^2 \mu^2}{8}, \ j=1,\dots ,N\ .
\]
Suppose first that $s/4 < \tilde{s} < s/2$.  Note that now the drift of the log-likelihood ratio process is positive.  Now $z_1 \sim N\left( (\tilde{s} - \frac{s}{4}) \frac{a^2 s \mu^2}{2}, \frac{a^2 s^2 \mu^2}{4} \right)$.  From normality we still have $\E (z_1 | z_1 \geq 0) \geq \E (z_1 -c | z_1 \geq c), \ \forall c>0$.  Combining this with Wald's identity we get
\[
\E (N) \E (z_1) = \E (\bar{z}_N) \leq u + \E (z_1 | z_1 \geq 0) \ ,
\]
where $\bar{z}_N=\sum_{j=1}^N z_j$.  Denoting $\xi \sim N(0,1)$ we also have
\begin{align*}
\E (z_1 | z_1 \geq 0) & \leq 2 \E (z_1 \1 \{ z_1 \geq 0 \} ) \\
& \leq \left( \tilde{s} - \frac{s}{4} \right) \frac{a^2 s \mu^2}{2} + 2 \E \left( \frac{a s \mu}{2} \xi \1 \{ \xi \geq - \left( \tilde{s} - \frac{s}{4}\right) \mu \} \right) \\
& \leq a s \mu \left( \left( \tilde{s} - \frac{s}{4}\right) \frac{a \mu}{2} +1 \right) \ .
\end{align*}
Plugging this in, and using that $\E (z_1) \geq \frac{a^2 s \mu^2}{2}$ we get
\[
a^2 \E (N) \leq \frac{2}{s \mu^2} u + \frac{2a}{\mu} \left( \left( \tilde{s} - \frac{s}{4}\right) \frac{a \mu}{2} +1 \right) \ .
\]
Hence in the limit $a \to 0$ we get
\[
a^2 \E (N) \leq \frac{2}{s \mu^2} \log \frac{1- \beta}{\alpha} \leq \frac{2}{s \mu^2} \log \frac{1}{\alpha} \ .
\]
We can treat the case $0 < \tilde{s} < s/4$ in a similar fashion.
\end{proof}

\begin{remark}
When $\tilde{s} = s/4$ the argument of the proof breaks down, because of ties when $s$ is divisible by 4.  However this is only a technical issue that can be simply circumvented by choosing the bins to be of size $s/2 -1$, for instance.
\end{remark}

Now we are ready to upper bound the expected sensing energy used by the procedure.  Given $\alpha$ and $\beta$ we can deal with the search phase and by Lemma~\ref{lemma:SLRT} we can deal with the refinement phase given $\alpha' ,\beta'$ and $|\vec{P}|$.

Note that we have
\[
\E_S (|\vec{P}|) \leq 3ks + \frac{3s}{2} \sum_{i: \ \vec{P}^{(i)} \nsubseteq S} \alpha \ .
\]
Thus choosing $\alpha = \varepsilon /6n$ we have $\E_S (|\vec{P}|) \leq 3ks + \varepsilon /2 \leq 4ks$.

By denoting the part of the sensing matrix $A$ corresponding to the search and refinement phases by $A_{search}$ and $A_{refinement}$ respectively, we have
\begin{align}\label{eqn:int_performance}
\E_S (\| A \|_F ) & \leq \E_S ( \| A_{search} \|_F^2 ) + \E_S \left( \E_S ( \| A_{refinement} \|_F^2 \big| |\vec{P}| ) \right) \nonumber \\
& \leq \frac{16 n}{s^2 \mu^2} \log \frac{2\sqrt{2}ks}{\varepsilon} + \frac{4k}{s \mu^2} \log \frac{6n}{\varepsilon} + \frac{2k}{\mu^2} \log \frac{6n}{\varepsilon} \nonumber \\
& + \frac{8ks}{\mu^2} \log \frac{4n}{\varepsilon} \ .
\end{align}

When $|S| \ll n$ the first term dominates the bound above.  Using this and combining the above with \eqref{eqn:budget} we arrive at the following.

\begin{proposition} \label{prop:int}
Consider the class of $k$ disjoint $s$-intervals and suppose $\frac{n}{\log 4n} \geq ks^3$.  Then the above estimator satisfies \eqref{eqn:budget} and \eqref{eqn:goal} whenever
\[
\mu \geq \sqrt{\frac{30 n}{s^2 m} \log \frac{2\sqrt{2}ks}{\varepsilon}} \ .
\]
\end{proposition}
\begin{remark}
The condition on the sparsity in the proposition is needed to ensure that the term corresponding to the search phase in \eqref{eqn:int_performance} becomes dominant.  By performing the refinement phase in a more sophisticated way one can relax that condition.  For instance using $k$ binary searches to find the left endpoint of the intervals the sparsity condition becomes $\frac{n}{\log 6n} \geq ks^2 \log s$.  We expect this to be essentially the best condition one can hope for, as the lower bounds of Section~\ref{sec:sigstrength_lower} show that the first term in \eqref{eqn:int_performance} is unavoidable.
\end{remark}

The bound of Proposition~\ref{prop:int} matches the lower bound in Section~\ref{sec:sigstrength_lower}, hence in this sparsity regime the procedure above is optimal apart from constants.


\subsubsection{Unions of $s$-stars}\label{sec:star}

Let the components of $\vec{x}$ be in one-to-one correspondence to edges of a complete graph $G=(V,E)$.  Let $e_i \in E$ denote the edge corresponding to component $\vec{x}_i$, and for a vertex $v \in V$ and edge $e \in E$ let $v \in e$ denote that $e$ is incident with $v$.  We call a support set $S \subset \{ 1,\dots, n\}$ an $s$-star if $|S|=s$ and $\exists v \in V: \ \forall i \in S: \ v \in e_i$.  Let $\C$ be the class of unions of $k$ disjoint $s$-stars.  In what follows we use the notation $|V|=p$.

The procedure for support estimation is very similar to that presented for $s$-intervals.  We introduce the procedure when $k=1$, but the idea can be carried through for larger $k$.  Consider the subsets $\vec{P}^{(i)}, \ i=1,\dots ,p$, defined as follows:
\[
\vec{P}^{(i)} = \left\{j\in \{ 1,\dots ,n\} : \ v_i \in e_j\right\}\ ,
\]
that is $\vec{P}^{(i)}$ contains all the components whose corresponding edges lie on the vertex $v_i$.  These subsets are not a partition of $\{1,\ldots,n\}$ as they are not disjoint.  Nonetheless we know that
\[
|\vec{P}^{(i)} \cap S| \in \{ 0,1,s \} \qquad \forall i=1,\dots ,p \ .
\]
We can use this to find the approximate location of $S$.  Thus in the search phase we test the hypotheses
\[
H_0^{(i)}: \ |\vec{P}^{(i)} \cap S| =1 \qquad \textrm{versus} \qquad H_1^{(i)}: \ |\vec{P}^{(i)} \cap S| =s \qquad i=1,\dots ,p \ .
\]
In words we test whether vertex $v_i$ is the center of the star or not for $i=1,\dots ,p$.  Note that when vertex $v_i$ is not the center of the star we have $|\vec{P}^{(i)} \cap S| \in \{ 0,1 \}$.  By specifying $H_0^{(i)}$ as above we ensure that if $|\vec{P}^{(i)} \cap S| = 0$ both the probability of error and the expected number of steps of the SLRT will be smaller than if $|\vec{P}^{(i)} \cap S| = 1$, due to the monotonicity of the likelihood ratio.

Again we use independent SLRTs for the tests with common type I and type II error probabilities $\alpha, \beta$, where the details will be covered later.  Using these tests we can define $\vec{P}$, the output of the search phase, as the union of those $\vec{P}^{(i)}$ for which $H_1^{(i)}$ is accepted.  With the appropriate choices for $\alpha$ and $\beta$ we can ensure that with high probability $S \subset \vec{P}$ and that $|\vec{P}|$ is small.  In fact we would like to accept exactly one $H_1^{(i)}$.  Again the right choice for $\beta$ will ensure $\P_S (S \nsubseteq \vec{P})$ is small whereas the right choice of $\alpha$ ensures that $|\vec{P}|$ is small with high probability.  In the subsequent refinement phase we estimate $S$ within $\vec{P}$.  We do this using the same procedure as in the unstructured case with error probabilities $\alpha' ,\beta'$.  Finally the estimator $\hat{S}$ will be the collection of those components which were deemed active in the refinement phase.

Now we choose the error probabilities for the tests such that we can ensure \eqref{eqn:goal} for our procedure.  We have
\begin{align*}
\E_S \left( |\hat{S} \triangle S| \right) & \leq \E_S \left( |\hat{S} \triangle S| \ \big| S \nsubseteq \vec{P} \right) \P_S (S \nsubseteq \vec{P}) + \E_S \left( |\hat{S} \triangle S| \ \big| S \subseteq \vec{P} \right) \\
& \leq \E_S \left( \left.  |S \setminus \vec{P}| + \sum_{i \in \vec{P}: \ i \notin S} \alpha' + \sum_{i \in \vec{P}: \ i \in S} \beta' \ \right| S \nsubseteq \vec{P} \right) \beta \\
& + n \alpha' + s \beta' \ .
\end{align*}
Thus the choices $\beta = \varepsilon /4s$ and $\alpha' = \varepsilon /4n ,\beta' = \varepsilon /4s$ suffice.  As noted before, the choice of $\alpha$ will influence the size of $\vec{P}$ and will be discussed later.

To test $H_0^{(i)}$ versus $H_1^{(i)}$ we collect observations using the sensing vector $a \1_{\vec{P}^{(i)}}$ with an arbitrarily small $a$ and perform a SLRT such as the one in Lemma~\ref{lemma:SLRT_int}.  When there is no active component in $\vec{P}^{(i)}$ the drift of the likelihood-ratio process is smaller than if there was one active component by monotonicity of the likelihood ratio.  This results in the test terminating sooner in expectation than it would under $H_0^{(i)}$ and the probability of accepting $H_1^{(i)}$ is also smaller than the type I error probability $\alpha$.

We continue by upper bounding the expected sensing energy used by the procedure.  Again we have results similar to Lemma~\ref{lemma:SLRT_int} for the tests carried out in the search phase, and we can use Lemma~\ref{lemma:SLRT} to bound the energy used in the refinement phase.  Hence given $\alpha ,\beta ,\alpha' ,\beta'$ and $\vec{P}$ we can bound the total energy used by the procedure.  Also note that
\[
\E_S (|\vec{P}|) \leq p + p \sum_{i: \ \vec{P}^{(i)} \nsubseteq S} \alpha \ ,
\]
thus choosing $\alpha = \varepsilon /2n$ ensures $\E_S (|\vec{P}|) \leq 2 p$.

Using the notation $A_{search}$ and $A_{refinement}$ as before we get
\begin{align*}
\E_S (\| A \|_F ) & \leq \E_S ( \| A_{search} \|_F^2 ) + \E_S \left( \E_S ( \| A_{refinement} \|_F^2 \big| |\vec{P}| ) \right) \\
& \leq \frac{2p (p-1)}{(s-1)^2 \mu^2} \log \frac{4s}{\varepsilon} + \frac{2p}{(s-1)^2 \mu^2} \log \frac{4n}{\varepsilon} \\
& + \frac{4p}{\mu^2} \log \frac{4n}{\varepsilon} \ .
\end{align*}
When $s \ll n$ the first term dominates the bound.  Combining this with \eqref{eqn:budget} we get the following.

\begin{proposition} \label{prop:star}
Consider the class of $s$-stars and suppose $\frac{\sqrt{n}}{\log 4n} \geq s^2$.  Then the above estimator satisfies \eqref{eqn:budget} and \eqref{eqn:goal} whenever
\[
\mu \geq \sqrt{\frac{16 n}{(s-1)^2 m} \log \frac{4s}{\varepsilon}} \ .
\]
\end{proposition}

In Section~\ref{sec:sigstrength_lower} we show that the bound of Proposition~\ref{prop:star} is near optimal in this sparsity regime.  We also show there that the sparsity assumption in the proposition above is needed and is not an artifact of our method.

When $k>1$ ($S$ consists of two or more $s$-stars) similar arguments hold.  When $k \ll s$ it is possible to modify the procedure such that the search phase aims to find the center of the $k$ stars.  The modifications include setting $H_0{(i)}: \ |\vec{P}^{(i)} \cap S| =k$, and slightly changing $\alpha, \beta, \alpha' ,\beta'$ to account for the fact that there are more than one stars.  For instance choosing $\alpha ,\alpha'$ to be the same as before and setting $\beta = \beta' = \varepsilon /4ks$ we get the following.

\begin{proposition} \label{prop:union_star}
Consider the class of $k$ disjoint $s$-stars and suppose $k<s$ and $\frac{\sqrt{n}}{\log 4n} \geq k (s-k)^2$.  Then the modified estimator satisfies \eqref{eqn:budget} and \eqref{eqn:goal} whenever
\[
\mu \geq \sqrt{\frac{16 n}{(s-k)^2 m} \log \frac{4sk}{\varepsilon}} \ .
\]
\end{proposition}

We see is Section~\ref{sec:sigstrength_lower} that the bound above is near the optimal one when $k$ is much smaller than $s$.


\subsubsection{$s_r ,s_c$-submatrices}\label{sec:submat}

Let the components of $\vec{x}$ be in one-to-one correspondence to elements of a matrix $M$ with $n_r$ rows and $n_c$ columns (and let $n=n_r\cdot n_c$).  We call a set $S \subset \{ 1,\dots ,n\}$ an $s_r ,s_c$-submatrix if the elements $m_i \in M$ corresponding to the components $i \in S$ form an $s_r \times s_c$ submatrix in $M$.  Let $\C$ be the class of all $s_r ,s_c$-submatrices in $\vec{x}$.  Suppose without loss of generality that $s_r \geq s_c$ and recall that the number of non-zero components of $\vec{x}$ is simply $s=s_r \cdot s_c$.

One possible way to estimate $S$ is to first find the active columns in the search phase and then focus on one or more active columns in the refinement phase to find the active rows.  Let $\vec{c}^{(i)}$ denote the $i$th column of $\vec{x}$, $i=1,\dots ,n_c$.  In order to find the active columns we need to decide between
\[
H_0^{c(i)} : \ |\vec{c}^{(i)} \cap S| = 0 \qquad \textrm{versus} \qquad H_1^{c(i)} : \ |\vec{c}^{(i)} \cap S|=s_r \qquad i=1,\dots ,n_c \ .
\]
To do this we perform independent SLRTs with type I and type II error probabilities $\alpha$ and $\beta$ respectively for every $i=1,\dots ,n_c$.  At the end of the search phase we return $\vec{P}$, which is the union of columns $c^{(i)}$ for which $H_1^{c(i)}$ was accepted.  Choosing $\alpha ,\beta$ appropriately ensures that with high probability $\vec{P}$ contains all the active columns and only those.  In the refinement phase we test if row $j$ of $\vec{P}$ is active or not using a similar method as above, with error probabilities $\alpha' ,\beta'$ for every $j=1,\dots ,n_r$.  In particular the tests are formulated as
\[
H_0^{r(j)} : \ |(\vec{r}^{(j)} \cap \vec{P}) \cap S| = 0 \qquad \textrm{versus} \qquad H_1^{r(j)} : \ |(\vec{r}^{(j)} \cap \vec{P}) \cap S|=s_c \qquad j=1,\dots ,n_r \ ,
\]
where $\vec{r}^{(j)}$ denotes the $j$th row of $\vec{x}$, $j=1,\dots ,n_r$.  Finally our estimate $\hat{S}$ are those elements that are in a row and column that were both deemed active.

Now we choose the error probabilities $\alpha ,\beta ,\alpha' ,\beta'$.  Now we simply have
\[
\E_S (|\hat{S} \triangle S|) \leq n \alpha + s \beta + n \alpha' + s \beta' \ ,
\]
as every type I error in the search phase can result in at most $n_r$ errors in $\hat{S}$ and there can be at most $n_c$ type I errors in the search phase, whereas a type II error can produce at most $s_r$ errors in the end and there are $s_c$ possibilities to make such an error.  A similar argument holds for tests in the refinement phase.  Hence the choices $\alpha = \alpha' = \varepsilon / 4n$ and $\beta = \beta' = \varepsilon /4s$ ensure \eqref{eqn:goal}.

We move on to bounding the expected energy used by the procedure.  To test the $i$th hypothesis in the search phase we collect measurements using sensing vector $a \1_{\vec{c}^{(i)}}$ with $a$ arbitrarily small for all $i=1,\dots ,n_c$ and perform a SLRT similar to that described in the previous cases.  To perform the $j$th SLRT of the refinement phase we collect measurements of the form $a \1_{r^{(j)} \cap \vec{P}}$ using an arbitrarily small $a$.  For these tests we have results identical to Lemmas~\ref{lemma:SLRT_int}~and~\ref{lemma:bad_SLRT}.  Also for the number of columns in $\vec{P}$ denoted by $\tilde{n}_c$ we have
\[
\E_S (\tilde{n}_c) \leq s_c + n_c \alpha \leq 2 s_c \ .
\]
Putting everything together yields
\begin{align*}
\E_S (\| A \|_F ) & \leq \E_S ( \| A_{search} \|_F^2 ) + \E_S \left( \E_S ( \| A_{refinement} \|_F^2 \big| |\vec{P}| ) \right) \\
& \leq \frac{2n}{s_r^2 \mu^2} \log \frac{4s}{\varepsilon} + \frac{2 n_r s_c}{s_r^2 \mu^2} \log \frac{4n}{\varepsilon} \\
& + \frac{4 n_r}{s_c \mu^2} \log \frac{4n}{\varepsilon} \ .
\end{align*}
When $s \ll n$ the first term dominates the bound above.  Combining this with \eqref{eqn:budget} yields the following.

\begin{proposition} \label{prop:submat1}
Consider the class of $s_r ,s_c$-submatrices and suppose $\frac{n_c}{\log 4n} \geq \frac{s_r^2}{s_c}$.  Then the estimator above satisfies \eqref{eqn:budget} and \eqref{eqn:goal} whenever
\[
\mu \geq \sqrt{\frac{8 n}{s_r^2 m} \log \frac{4s}{\varepsilon}} \ .
\]
\end{proposition}

Note that the condition on the sparsity in the proposition above is not very strict.  Consider square submatrices within square matrices so that we have $n_r = n_c = \sqrt{n}$ and $s_r = s_c = \sqrt{s}$.  Then the condition becomes $\frac{\sqrt{n}}{\log 4n} > \sqrt{s}$, which would be automatically fulfilled if there was no logarithmic term on the left.  We see in Section~\ref{sec:sigstrength_lower} that in some sparsity regimes the bound above matches the lower bounds we derive, thus in those regimes this procedure is near optimal.  However, in what follows we slightly modify the procedure above to have better performance for submatrices that are more sparse than the ones required in the proposition above.  This combined with the results of Section~\ref{sec:sigstrength_lower} shows that the best performance we can hope for depends on the sparsity in a non-trivial manner in the case of submatrices.

Note that in principle it is enough to find a single active column in the search phase, as accurately estimating components within \emph{any} active column will yield the identity of all the active rows and similarly estimating components within \emph{any} active row yields the active columns.  This motivates the following modification of the above procedure: return a single active column in the search phase, then focus on that column to find the active rows and finally focus on one active row to find the active columns.  To do this we retain most of the algorithm choices done in the earlier approach, but choose a different $\alpha$ and $\beta$.

Ideally we would like to accept $H_1^{c(i)}$ for exactly one active column, so our choices for $\alpha ,\beta$ will be made accordingly.  In the refinement phase we choose a column randomly from the ones that were deemed active and locate the active components within that column, using the same procedure as in the unstructured case.  This gives us the active rows.  Finally we choose a row deemed active, and find all the active components within that row to find the active columns.  Throughout the refinement phase we set type I and type II error probabilities to be $\alpha' ,\beta'$.  With the right choices for the error probabilities, this procedure outperforms the previous one in certain sparsity regimes.

First we need to choose the error probabilities for the tests.  We can write
\begin{align*}
\E_S (|\hat{S} \triangle S|) & \leq 2s \P_S (\vec{P} = \emptyset ) + \left( 2n \alpha' +2s \right) \P_S ( \exists \vec{c}^{(i)} \subset \vec{P}: \ \vec{c}^{(i)} \cap S = \emptyset ) + \left( 2n \alpha' + 2s \beta' \right) \\
& \leq 2s \beta^{s_c} + (2n \alpha' +2s) n_c \alpha + ( 2n \alpha' + 2s \beta' ) \ .
\end{align*}
Thus the conservative choices $\alpha = \varepsilon /16 n^2 ,\beta = \sqrt[s_c]{\varepsilon /8s} ,\alpha' = \varepsilon /8n ,\beta' = \varepsilon /8s$ ensure \eqref{eqn:goal}.

Now we can move on to calculate the expected sensing energy used by the procedure.  The same way as before we have
\begin{align*}
\E_S (\| A \|_F ) & \leq \E_S ( \| A_{search} \|^2_F ) + \E_S ( \| A_{refinement} \|^2_F ) \\
& \leq \frac{2n}{s_c s_r^2 \mu^2} \log \frac{8s}{\varepsilon} + \frac{4 n_r s_c}{s_r^2 \mu^2} \log \frac{4n}{\varepsilon} \\
& + \frac{4 \max \{ n_r ,n_c \} }{\mu^2} \log \frac{8n}{\varepsilon} \ .
\end{align*}
Combining the above with \eqref{eqn:budget} and using that when $s \ll n$ the first term dominates and we arrive to the following result.

\begin{proposition} \label{prop:submat2}
Consider the class of $s_r ,s_c$-submatrices and suppose $\frac{\min \{ n_r ,n_c \} }{\log 8n} \geq s_c s_r^2$.  Then the estimator above satisfies \eqref{eqn:budget} and \eqref{eqn:goal} whenever
\[
\mu \geq \sqrt{\frac{10 n}{s_c s_r^2 m} \log \frac{8s}{\varepsilon}} \ .
\]
\end{proposition}

The condition on the sparsity in the proposition above is stronger than that in Proposition~\ref{prop:submat1}.  On the other hand the bound for $\mu$ is smaller.  This shows that in sparser regimes it is indeed possible to outperform the procedure of Proposition~\ref{prop:submat1}, hinting that the sparsity regime non-trivially influences the best possible performance of adaptive support recovery procedures in the case of sub-matrices.  For instance considering square matrices when $n_r = n_c = \sqrt{n}$ and $s_r = s_c = \sqrt{s}$, the condition above reads $\frac{\sqrt{n}}{2 \log 4n} > \sqrt{s^3}$ which is slightly stronger than that of Proposition~\ref{prop:submat1}.


\subsection{Lower bounds} \label{sec:sigstrength_lower}

We turn our attention to the fundamental limits of recovering the support of structured sparse signals using compressive measurements by \emph{any} adaptive sensing procedure.  We consider both the non-adaptive sensing and adaptive sensing settings.  Some of the lower bounds presented below consider the probability of error $\P_S (\hat{S} \neq S)$ as the error metric.  Note that this is more forgiving than $\E_S (|\hat{S} \triangle S|)$, hence lower bounds with the former metric in mind apply as lower bounds with the latter metric as well.

\subsubsection{Non-Adaptive Sensing}

First we consider the non-adaptive compressive sensing setting.  Comparing these lower bounds with the performance bounds of the previous section illustrates the gains adaptivity provides in the various cases.  We do not make any claim on whether these lower bounds are tight or not, as these serve mostly for comparison between adaptive and non-adaptive sensing.  The lower bounds presented for the non-adaptive case consider $\P_S (\hat{S} \neq S)$ as the error metric.  As highlighted above, these are therefore valid lower bounds for the procedures with the expected Hamming-distance as the error metric.  Furthermore, for certain classes ($s$-sets, $s$-intervals) there exist procedures satisfying $\max_{S\in \C} \E_S (|\hat{S} \triangle S|)\leq \varepsilon$ with performance matching the lower bounds below.

In the non-adaptive sensing setting we need to define sensing actions before any measurements are taken.  That means the sensing matrix $A$ is specified prior to taking any observations.  This does not exclude the possibility that $A$ is random, but it has to be generated before any observations are made.

All the bounds presented here are based on Proposition 2.3 in in \cite{Tsybakov_2009}, which states
\begin{lemma}[Proposition 2.3 of \cite{Tsybakov_2009}] \label{lemma:tsybakov}
Let $\P_0, \dots, \P_M$ be probability measures on $(\mathcal{X},\mathcal{A})$ and let $\Psi: \mathcal{X} \to \{ 0, \dots ,M \}$ be any $\mathcal{A}$-measurable function.  If
\[
\frac{1}{M} \displaystyle{\sum_{j=1}^{M}} D( \P_j \| \P_0 ) \leq t
\]
then
\[
\displaystyle{\max_{j=0, \dots ,M}} \P_j (\Psi \neq j) \geq \displaystyle{\sup_{0< \tau <1}} \left( \frac{\tau M}{1+ \tau M} \left( 1+ \frac{t + \sqrt{t/2}}{\log \tau} \right) \right) \ .
\]
\end{lemma}

We can use this to get lower bounds in the following way.  Let $\P_0, \dots ,\P_M$ be the probability measures induced by sampling $\vec{x}$ with sensing matrix $A$, when the support set is $S_0, \dots ,S_M$ respectively, where $S_i \in \mathcal{C}$.  Now note that
\begin{align} \label{eqn:KL_bound}
\lefteqn{D(\P_j \| \P_0 )  = \E_0 \left( \sum_k \log \frac{\mathrm{d}\P_0 (Y_k | A_k)}{\mathrm{d}\P_j (Y_k | A_k)} \right)} \nonumber \\
& = \sum_k \E \left( \E_0 \left( \left.  - \frac{1}{2} \left( (Y_k - \mu <A_k,\1_{S_0}>)^2 - (Y_k - \mu <A_k,\1_{S_j}>)^2 \right| A \right) \right) \right) \nonumber \\
& = \sum_k \E \left( \E_0 \left( \left.  \frac{1}{2} \left( \mu^2 ( <A_k,\1_{S_j}>^2 -<A_k,\1_{S_0}>^2 ) -2\mu Y_k <A_k,\1_{S_j} - \1_{S_0}> \right) \right| A \right) \right) \nonumber \\
& = \frac{\mu^2}{2} \E \left( \sum_k \left( <A_k,\1_{S_j}>^2 + <A_k,\1_{S_0}>^2 -2 <A_k,\1_{S_j}> <A_k,\1_{S_0}> \right) \right) \nonumber \\
& = \frac{\mu^2}{2} \E \left( \sum_k <A_k, \1_{S_j} - \1_{S_0}>^2 \right) \nonumber \\
& \leq \frac{\mu^2}{2} \E \left( \sum_k |S_0 \triangle S_j | \sum_{i \in S_0 \triangle S_j} A_{k,i}^2 \right) \nonumber \\
& = \frac{\mu^2}{2} |S_0 \triangle S_j | \sum_{i \in S_0 \triangle S_j} a_i^2 \ ,
\end{align}
where $A_{k,j}$ is the $(k,j)$th element of the sensing matrix $A$, $a_i^2$ denotes $\E( \sum_k A_{k,i}^2)$, and in the second to last step we use Jensen's inequality.

Now consider the right side of Lemma~\ref{lemma:tsybakov} and set $\tau = 1/M$.  To make the bound more transparent suppose $1 \leq (1- 2\varepsilon ) \log M$, which is essentially always satisfied if $M$ is large enough and $\varepsilon \in (0,1/2)$.  This way we arrive to the inequality
\begin{equation} \label{eqn:nonadapt_lower}
2t \geq (1- 2\varepsilon ) \log M \ .
\end{equation}
Choosing the sets $S_0 ,\dots ,S_M$ and using inequality \eqref{eqn:KL_bound} to bound the average KL distance, we can use the above inequality to get lower bounds for $\mu$.  These choices will be specific to the classes we are considering.

\begin{remark}
In the following statements we require $n$ to be divisible by $s$.  This condition is merely for technical convenience, and can be easily dropped at the expense of a cumbersome presentation.
\end{remark}


\begin{proposition}[$s$-sets] \label{prop:vanilla_nonadapt_lower}
Let $\mathcal{C}$ be the class of $s$-sets and suppose $n/s$ is an integer.  If there is a non-adaptive estimator $\hat{S}$ that satisfies \eqref{eqn:budget} and $\P_S (\hat{S} \neq S) \leq \varepsilon \ \forall S \in \mathcal{C}$ then
\[
\mu \geq \sqrt{(1- 2\varepsilon ) \frac{n}{4m} \log (n-s)} \ .
\]
\end{proposition}
\begin{proof}
Let $S_0 \in \mathcal{C}$ be arbitrary.  Partition $\{ 1,\dots ,n \}$ into $s$ bins of equal size denoted by $\vec{P}^{(1)},\dots ,\vec{P}^{(s)}$ such that each bin contains exactly one element of $S_0$.  Let $s_i = S_0 \cap \vec{P}^{(i)}, \ i=1,\dots ,s$.  Now consider the sets $S_1,\dots ,S_M$ that we get by modifying exactly one element of $S_0$ in the following way: pick one element of $S_0$ denoted by $s_i$ and swap it with some other element in $\vec{P}^{(i)}$ thus changing the position of the active component within $\vec{P}^{(i)}$.  We can generate $M=n-s$ sets in the previous manner.  From \eqref{eqn:KL_bound} we have that
\[
\frac{1}{M} \sum_{j=1}^{M} D(\P_j \| \P_0 ) \leq \frac{1}{M} \mu^2 \sum_{j=1}^{M} \sum_{i \in S_0 \triangle S_j} a_i^2 = \frac{1}{n-s} \mu^2 \left( \sum_{i=1}^{n} a_i^2 + \frac{n-2s}{s} \sum_{i \in S_0} a_i^2 \right) \ .
\]
Now note that by the total energy constraint \eqref{eqn:budget} we have
\[
\sum_{i=1}^{n} a_i^2 \leq m \ .
\]
Also note that given $A$ we can always choose $S_0$ to be the one that is the most difficult to distinguish from the other sets $S_1 ,\dots ,S_M$.  That is we have to solve
\[
\max_{A: \ \| A\|_F \leq m} \ \min_{S_0 \in \mathcal{C}} \ \sum_{i \in S_0} a_i^2 \ .
\]
This implies $\displaystyle{\sum_{i \in S_0}} a_i^2 \leq sm/n$.  Combining what we have yields
\[
\frac{1}{M} \sum_{j=1}^{M} D(\P_j \| \P_0 ) \leq \frac{1}{n-s} \left( 1+ \frac{n-2s}{n} \right) m \mu^2 \leq \frac{2m}{n} \mu^2 \ .
\]
Using this with \eqref{eqn:nonadapt_lower} concludes the proof.
\end{proof}


\begin{proposition}[Unions of $s$-intervals] \label{prop:int_nonadapt_lower}
Let $\mathcal{C}$ be the class of unions of $k$ disjoint $s$-intervals and suppose $n/s$ is an integer.  If there is a non-adaptive estimator $\hat{S}$ that satisfies \eqref{eqn:budget} and $\P_S (\hat{S} \neq S) \leq \varepsilon \ \forall S \in \mathcal{C}$ then
\[
\mu \geq \sqrt{(1- 2\varepsilon ) \frac{n-(k-1)s}{4 s^2 m} \log (\frac{n}{s}-k)} \ .
\]
\end{proposition}
\begin{proof}
Partition $\{ 1,\dots, n\}$ into consecutive intervals of size $s$ denoted by $S^{(1)},\dots ,S^{(n/s)}$.  Now consider the subclass whose elements are unions of the first $k-1$ intervals $S^{(1)},\dots ,S^{(k-1)}$ and some other interval $S^{(i)}$.  Formally, $\mathcal{C}' = \{ S \in \mathcal{C} : \ S= S^{(i)} \cup \left( \bigcup_{j=1}^{k-1} S^{(j)} \right), \ i=k,\dots ,n/s \}$.  This way we effectively reduced this problem to finding one interval in a slightly smaller vector.  Let $S_0 \in \mathcal{C}'$ be arbitrary and let $S_1 ,\dots ,S_M$ be all the other elements of $\mathcal{C}'$, so $M=n/s-k$.  Let $\tilde{S}_0 = S_0 \setminus \cup_{j=1}^{k-1} S^{(j)}$.  From \eqref{eqn:KL_bound} we have
\[
\frac{1}{M} \sum_{j=1}^{M} D(\P_j \| \P_0 ) \leq s \mu^2 \frac{1}{M} \sum_{j=1}^{M} \sum_{i \in S_0 \triangle S_j} a_i^2 = \frac{s^2 \mu^2}{n-ks} \left( \sum_{i=(k-1)s+1}^{n} a_i^2 + \frac{n-(k+1)s}{s} \sum_{i \in \tilde{S}_0} a_i^2 \right) \ .
\]
Again, from \eqref{eqn:budget} and the fact that we can choose $S_0 \in \mathcal{C}'$ after the sensing strategy has been determined we have
\[
\frac{1}{M} \sum_{j=1}^{M} D(\P_j \| \P_0 ) \leq \frac{1}{n-ks} \left( 1 + \frac{n-(k+1)s}{n-(k-1)s} \right) s^2 m \mu^2 \leq \frac{2 s^2 m}{n-(k-1)s} \mu^2 \ .
\]
Using this with \eqref{eqn:nonadapt_lower} concludes the proof.
\end{proof}


\begin{proposition}[$s$-stars] \label{prop:star_lower}
Let $\mathcal{C}$ be the class of $s$-stars and suppose $p/s$ is an integer.  If there is a non-adaptive estimator $\hat{S}$ that satisfies \eqref{eqn:budget} and $\P_S (\hat{S} \neq S) \leq \varepsilon \ \forall S \in \mathcal{C}$ then
\[
\mu \geq \sqrt{(1- 2\varepsilon ) \frac{n}{2 m} \log (\sqrt{2n}-s-1)} \ .
\]
\end{proposition}
\begin{proof}
Consider the $p-1$ edges of the complete graph of $p$ vertices which share a common vertex $j$.  Denote this set of edges by $E_j$.  The $s$-stars whose center is vertex $j$ form a class of $s$-sets on $E_j$.  So we can do the same construction on this set of edges as in Proposition~\ref{prop:vanilla_nonadapt_lower} to get
\[
\frac{1}{M} \sum_{j=1}^{M} D(\P_j \| \P_0 ) \leq \frac{1}{p-1-s} \mu^2 \left( \sum_{i \in E_j} a_i^2 + \frac{p-1-2s}{s} \sum_{i \in S_0} a_i^2 \right) \ .
\]
Now note that we can choose any star to be $S_0$ which implies $\sum_{i \in S_0} a_i^2 \leq sm/n$ and $\sum_{i \in E_j} a_i^2 \leq (p-1)m/n$ yielding
\[
\frac{1}{M} \sum_{j=1}^{M} D(\P_j \| \P_0 ) \leq \frac{2m}{n} \mu^2 \ .
\]
The statement now follows from \eqref{eqn:nonadapt_lower} and that $p > \sqrt{2n}$.
\end{proof}

Considering unions of $k$ disjoint $s$-stars we can get a similar lower bound by considering a subclass where $k-1$ of the $s$-stars are fixed and only one can change, reducing the problem to finding one $s$-star.


\begin{proposition}[$s$-submatrices]\label{prop:submat_nonadapt_lower}
Let $\mathcal{C}$ be the class of $s$-submatrices of a fixed size $s_c \times s_r$, and suppose both $n_c/s_c$ and $n_r/s_r$ are integers.  If there is a non-adaptive estimator $\hat{S}$ that satisfies \eqref{eqn:budget} and $\P_S (\hat{S} \neq S) \leq \varepsilon \ \forall S \in \mathcal{C}$ then
\[
\mu \geq \sqrt{(1- 2\varepsilon ) \frac{n}{4 m} \max \left\{ \frac{1}{s_r}\frac{n_c -s_c}{n_c} , \frac{1}{s_c}\frac{n_r -s_r}{n_r} \right\} \log \left( \max \{ n_r -s_r ,n_c -s_c \} \right) } \ .
\]
\end{proposition}
\begin{proof}
Let $S_0 \in \mathcal{C}$ be arbitrary.  Denote the indexes of the rows of $S_0$ by $r_1,\dots ,r_{s_r}$, and let $S_0^{(j)}$ denote the $j$th row of $S_0$.  Consider a partition of the indexes $\{ 1,\dots ,n_r \}$ into $\vec{r}^{(1)},\dots ,\vec{r}^{(s_r)}$ such that all of the are of the same size and $\vec{r}^{(j)}$ contains exactly one active row indexed by $r_j$ for every $j=1,\dots, r_{s_r}$.

Now let $S_1,\dots ,S_M$ be elements of $\mathcal{C}$ that we get by replacing exactly one row index of $S_0$ such that if we modify $r_j$, then the new row index is in $\vec{r}^{(j)}$.  There are $n_r - s_r$ such submatrices.  The same way as for the $s$-sets we get
\[
\frac{1}{M} \sum_{j=1}^{M} D(\P_j \| \P_0 ) \leq \frac{1}{n_r - s_r} \mu^2 \left( \sum_{(i,l): \ l \in C_0} a_{(i,l)}^2 + \frac{n_r - 2 s_r}{s_r} \sum_{(i,l) \in S_0} a_{(i,l)}^2 \right) \ ,
\]
where $C_0$ denotes the set of column indexes of $S_0$.  Again, the fact that we can choose an arbitrary $S_0 \in \mathcal{C}$ after the sensing strategy has been fixed results in the upper bound
\[
\frac{1}{M} \sum_{j=1}^{M} D(\P_j \| \P_0 ) \leq \frac{2 s_c m}{n} \frac{n_r}{n_r - s_r} \mu^2 \ .
\]
Plugging this into \eqref{eqn:nonadapt_lower} and rearranging gives a lower bound.  Repeating the same arguments for columns concludes the proof.
\end{proof}


\subsubsection{Adaptive Sensing}

Here we provide lower bounds considering the adaptive sensing framework.  Comparing these bounds with the performance bounds of Section~\ref{sec:sigstrength_proc} shows the near optimality of the procedures presented there.


\vspace{0.2cm}
\textbf{$s$-sets}

Adaptive sensing lower bounds for unstructured classes were proved in \cite{AS_Rui_2012}.  In that work lower bounds are derived by slightly broadening the class, which we state here for convenience.  Note that the fact that the following lower bound is valid for a slightly larger class than the class of $s$-sets does not cause a problem, see Remarks~\ref{remark:sparsity1} and \ref{remark:sparsity2}.  Let $\mathcal{C}_s$ denote the class of $s$-sets.  We have the following.

\begin{proposition} \label{prop:vanilla_lower}
Let $\mathcal{C} = \mathcal{C}_s \cup \mathcal{C}_{s-1}$, and suppose there exists an estimator $\hat{S}$ that satisfies \eqref{eqn:budget} and \eqref{eqn:goal}.  Then we have
\[
\mu \geq \sqrt{\frac{2(n-s+1)}{m} \left( \log \frac{s}{2 \varepsilon} + \log \frac{n-s+1}{n+1} \right)} \ .
\]
\end{proposition}

\begin{remark}\label{remark:sparsity2}
Note that the bound above holds for estimators for sets with sparsity $s$ or $s-1$.  The procedure presented in Section~\ref{sec:sigstrength_proc} works for this class of sets without any modifications.  Later on for the structured classes we rely on the proposition above to derive lower bounds, hence a similar comment applies in those cases as well.
\end{remark}


\textbf{$s$-intervals and unions of $s$-intervals}
\vspace{0.2cm}

For $s$-intervals we have multiple ways of deriving lower bounds, just as in the case of coordinate wise sampling studied in \cite{AS_structured_Rui_ET_2013}.  First we consider $\P_S (\hat{S} \neq S)$ as the error metric.  The following result is analogous to the lower bound in \cite{Balakrishnan_2012}, and the proof is included here for the sake of clarity.

\begin{proposition} \label{prop:int_lower1}
Let $\mathcal{C}$ be the class of $s$-intervals and suppose there is an estimator $\hat{S}$ satisfying \eqref{eqn:budget} and $\displaystyle{\max_{S \in \C}} \ \P_S (\hat{S} \neq S) \leq \varepsilon$.  Furthermore suppose $n/s$ is an integer.  Then
\[
\mu \geq (1-\varepsilon ) \sqrt{ \frac{n}{2 s^2 m}} \ .
\]
\end{proposition}
\begin{proof}
Consider the subclass of consecutive disjoint $s$-intervals
\[
\left\{ \{ 1,\dots ,s\} , \{ s+1,\dots ,2s \} ,\dots \{ n-s+1,\dots ,n\} \right\} \ .
\]
Partition this subclass into two subclasses of equal size denoted by $\mathcal{C}_1$ and $\mathcal{C}_2$.  Let $\pi_i$ denote the uniform distribution on the subclass $\mathcal{C}_i$ for $i=1,2$, and consider the two hypotheses $H_i : \ S \sim \pi_i, \ i=1,2$.  If there exists an estimator $\hat{S}$ satisfying \eqref{eqn:budget}, then there exists a test function $\Phi : D \rightarrow \{ 1,2 \}$ such that $\P_1 (\Phi (D) =2) + \P_2 (\Phi (D) = 1) \leq \varepsilon$, where $\P_i$ denotes the distribution of $D= \{Y_j ,A_j \}_{j=1,2,\dots}$ when $H_i$ is true, $i=1,2$.  Let $\P_0$ denote the distribution of $D$ when in fact $S = \emptyset$.  We have
\begin{align*}
\varepsilon & \geq \P_1 (\Phi (D) =2) + \P_2 (\Phi (D) = 1) \geq 1- TV(\P_1 ,\P_2 ) \\
& \geq 1- \left( TV(\P_0 ,\P_1 ) + TV(\P_0 ,\P_2 ) \right) = 1- 2 TV(\P_0 ,\P_1 ) \\
& \geq 1- \sqrt{2 KL(\P_0 ,\P_1 )} \ ,
\end{align*}
where $TV(.,.)$ denotes the total variation distance and $KL(.,.)$ denotes the Kullback-Leibler divergence of two distributions.  Now the goal is to upper bound $KL(\P_0 ,\P_1 )$.  Let $Y$ denote the observations $Y_1 ,Y_2 ,\dots$, and let $\P_S$ denote the distribution of $Y$ for a fixed support $S$.  We have
\begin{align*}
KL(\P_0 ,\P_1 ) & = \E_0 \left( \log \frac{\mathrm{d} \P_0 (D)}{\mathrm{d} \P_1 (D)} \right) = \E_0 \left( \log \frac{\mathrm{d} \P_0 (Y)}{\mathrm{d} \P_1 (Y)} \right) \\
& = - \E_0 \left( \log \frac{\mathrm{d} \P_1 (Y)}{\mathrm{d} \P_0 (Y)} \right) = - \E_0 \left( \log \frac{\E_{S \sim \pi_1} \left( \mathrm{d} \P_S (Y) \right)}{\mathrm{d} \P_0 (Y)} \right) \\
& \leq - \E_0 \left( \E_{S \sim \pi_1} \left( \log \frac{\mathrm{d} \P_S (Y)}{\mathrm{d} \P_0 (Y)} \right) \right) \\
& = - \E_0 \left( \E_{S \sim \pi_1} \left( -\frac{1}{2} \sum_{j=1}^{\infty} \left( (Y_j - \mu <A_j , \1_S>)^2 - Y_j^2 \right) \right) \right) \\
& = \frac{1}{2} \E_0 \left( \E_{S \sim \pi_1} \left( \sum_{j=1}^{\infty} \left( \mu^2 <A_j ,\1_S>^2 - 2 \mu <A_j ,\1_S> Y_j \right) \right) \right) \\
& = \frac{\mu^2}{2} \E_0 \left( \E_{S \sim \pi_1} \left( \sum_{j=1}^{\infty} A_j^T \1_S \1_S^T A_j \right) \right) \\
& = \frac{\mu^2}{2} \E_0 \left( \sum_{j=1}^{\infty} A_j^T \E_{S \sim \pi_1} \left( \1_S \1_S^T \right) A_j \right) \ ,
\end{align*}
where $\E_{S\sim \pi_1}$ is the expectation w.r.t.  $S$ when it is distributed according to $\pi_1$.  Now $\E_{S \sim \pi_1} \left( \1_S \1_S^T \right) = \frac{2s}{n} I'$ where $I' \in \mathbb{R}^{n \times n}$ is block diagonal with $n/2s$ blocks of size $s \times s$ consisting of all ones, and the rest of the matrix consists of zeros.  Thus we can continue as
\begin{align*}
KL(\P_0 ,\P_1 ) & \leq \frac{\mu^2}{2} \E_0 \left( \sum_{j=1}^{\infty} A_j^T \E_{S \sim \pi_1} \left( \1_S \1_S^T \right) A_j \right) \\
& = \mu^2 \frac{s}{n} \E_0 \left( \sum_{j=1}^{\infty} A_j^T I' A_j \right) = \mu^2 \frac{s}{n} \E_0 \left( \sum_{j=1}^{\infty} < A_j ,I' A_j > \right) \\
& \leq \mu^2 \frac{s}{n} \E_0 \left( \sum_{j=1}^{\infty} |< A_j ,I' A_j >| \right) \\
& \leq \mu^2 \frac{s}{n} \E_0 \left( \sum_{j=1}^{\infty} \| A_j \|_2 \| I' A_j \|_2 \right) \\
& \leq \mu^2 \frac{s}{n} \E_0 \left( \sum_{j=1}^{\infty} \| A_j \|_2^2 \| I' \|_2 \right) \\
& \leq \mu^2 \frac{m s^2}{n} \ ,
\end{align*}
where $\| I' \|_2$ is the matrix norm of $I'$ induced by the Euclidean norm and the last step follows from $\| I' \|_2 \leq s$ and \eqref{eqn:budget}.  Thus we arrive at the inequality
\[
\varepsilon \geq 1- \sqrt{2 \mu^2 \frac{m s^2}{n}} \ ,
\]
from which the statement follows.
\end{proof}

In the previous bound the dependence on $\epsilon$ is clearly loose.  When considering the Hamming distance as the error metric, we can also get lower bounds by slightly broadening the class.  We cover this by considering the case of unions of $k$ disjoint $s$-intervals, which as a special case contains the class of $s$-intervals when $k=1$.  We broaden this class by adding unions of $k-1$ disjoint $s$-intervals as well.

\begin{proposition} \label{prop:int_lower2}
Let $\mathcal{C}$ be the class of unions of $k$ or $k-1$ disjoint $s$-intervals with $k>0$ fixed, and suppose $n/s$ is an integer.  Suppose there is an estimator satisfying \eqref{eqn:budget} and $\displaystyle{\max_{S \in \mathcal{C}}} \ \E_S \big( d(\hat{S},S) \big) \leq \varepsilon$.  Then
\[
\mu \geq \sqrt{\frac{2 \big( n-s(k-1) \big)}{s^2 m} \left( \log \frac{ks}{8 \varepsilon} + \log \frac{n-s(k-1)}{n+s} \right)} \ .
\]
\end{proposition}
\begin{proof}
Partition $\{ 1,\dots ,n \}$ into consecutive disjoint $s$-intervals denoted by $S^{(1)} ,\dots ,S^{(n/s)}$, that is $S^{(d)} = \{ (d-1)s +1,\dots ,ds \}$, and consider the subclass $\mathcal{C}'$ of $\mathcal{C}$ consisting of all the sets in $\mathcal{C}$ that can be written in the form $\cup \ S^{(d)}$.  This subclass is similar to a general sparse class of sparsity $k$ or $k-1$ with the intervals $S_d$ playing the role of the components.  This is exactly what we wish to formalize, and then use Proposition~\ref{prop:vanilla_lower}.

Clearly $\displaystyle{\max_{S \in \mathcal{C}'}} \ \E_S \big( d(\hat{S},S) \big) \leq \varepsilon$.  Using $\hat{S}$ we can construct an estimator $\tilde{S}$ which only takes values of the form $\cup \ S^{(d)}$, and has the property $\displaystyle{\max_{S \in \mathcal{C}'}} \ \E_S \big( d(\tilde{S},S) \big) \leq 4 \varepsilon$.  For instance let $\tilde{S}$ be such that for every $d=1,\dots ,n/s : \ S^{(d)} \subset \tilde{S}$ if and only if $|\hat{S} \cap S^{(d)}| \geq s/2$.  The expected Hamming-distance for such estimators can be written as
\[
\E_S \left( d(\tilde{S} ,S) \right) = s \sum_{d=1}^{n/s} \P_S \left( \1 \{ S^{(d)} \subset \tilde{S} \} \neq \1 \{ S^{(d)} \subset S \} \right) \ .
\]

The measurements $Y_j , j=1,2,\dots$ can be written in the following form
\[
Y_j = <A_j ,\vec{x}> + W_j = \mu \sum_{i \in S} a_{i,j} + W_j = s \mu \sum_{S^{(d)} \in S} \frac{1}{s} \sum_{i \in S^{(d)}} a_{i,j} + W_j \ .
\]
Also from Jensen's inequality we have
\[
\sum_{j=1}^\infty \sum_{d=1}^{n/s} \left( \frac{1}{s} \sum_{i \in S^{(d)}} a_{i,j} \right)^2 \leq \sum_{j=1}^\infty \sum_{d=1}^{n/s} \frac{1}{s} \sum_{i \in S^{(d)}} a_{i,j}^2 = \frac{1}{s} \sum_{j=1}^\infty \sum_{d=1}^{n/s} \sum_{i \in S^{(d)}} a_{i,j}^2 \leq \frac{m}{s} \ .
\]

Therefore the problem can be viewed as estimating a general sparse support set.  The sparsity is either $k$ or $k-1$, the length of the vector is $n/s$, the signal strength is $s \mu$, the total sensing budget is $m/s$ and the desired accuracy in expected Hamming-distance is $4 \varepsilon /s$.  From Proposition~\ref{prop:vanilla_lower} we have
\[
s \mu \geq \sqrt{\frac{2(n/s-k+1)}{m/s} \left( \log \frac{ks}{8 \varepsilon} + \log \frac{n/s-k+1}{n/s+1} \right)} \ ,
\]
which concludes the proof.
\end{proof}


\textbf{$s$-stars and unions of $s$-stars}
\vspace{0.2cm}

For these classes exactly the same arguments follow as were used for $s$-intervals and unions of $s$-intervals.  The only thing that needs to be altered is that instead of disjoint $s$-intervals we use disjoint $s$-stars.  The difference this makes is that whereas before the new problem dimension became $n/s$, since the entire signal vector could be covered by disjoint intervals, the same can not be said when considering $s$-stars.

Let $N(p,s)$ denote the number of disjoint $s$-stars that can be packed in a complete graph with $p$ vertices.  We can easily check that the following inequality holds (see Lemma~2 in \cite{AS_structured_Rui_ET_2013})
\[
N(p,s) \geq \frac{p(p-1-s)}{2s} \ .
\]
The left hand side is approximately $n/s$ when the signal is sparse, thus essentially the same results hold as in the case of unions of intervals.  Thus the analogue of Proposition~\ref{prop:int_lower1} for $s$-stars is the following.

\begin{proposition} \label{prop:star_lower1}
Let $\mathcal{C}$ be the class of $s$-stars and suppose there is an estimator $\hat{S}$ satisfying \eqref{eqn:budget} and $\displaystyle{\max_{S \in \C}} \ \P_S (\hat{S} \neq S) \leq \varepsilon$.  Then
\[
\mu \geq (1-\varepsilon ) \sqrt{ \frac{N(p,s)}{2 s m}} \ .
\]
\end{proposition}

\begin{remark}
When $s \ll n$ the bound above scales as $(1-\varepsilon ) \sqrt{ \frac{n}{s^2 m}}$.
\end{remark}

We also have an analogue of Proposition~\ref{prop:int_lower2} for the case of multiple stars.

\begin{proposition} \label{prop:star_lower2}
Let $\mathcal{C}$ be the class of unions of $k$ or $k-1$ disjoint $s$-stars.  Suppose there is an estimator satisfying \eqref{eqn:budget} and $\displaystyle{\max_{S \in \mathcal{C}}} \ \E_S \big( d(\hat{S},S) \big) \leq \varepsilon$.  Then
\[
\mu \geq \frac{1}{s} \sqrt{\frac{2 \big( N(p,s)-k+1 \big)}{m/s} \left( \log \frac{ks}{8 \varepsilon} + \log \frac{N(p,s)-k+1}{N(p,s)+1} \right)} \ .
\]
\end{proposition}

\begin{remark}
When $s \ll n$ the bound above scales as $\sqrt{ \frac{n}{s^2 m} \log \frac{ks}{\varepsilon}}$.
\end{remark}

We also present another simple lower bound that illustrates that the assumption on the sparsity in Proposition~\ref{prop:star} requiring approximately that $s^4 \leq n$ is needed and is not only an artifact of our method.

Consider a setting where the support set is a star of size $s$ or $s-1$.  Now consider the sub-problem of estimating the support of such a star when the center of the star is given by an oracle.  This is an unstructured problem on a vector of size $p-1$.  Hence we can directly apply Proposition~\ref{prop:vanilla_lower} to get the following result.

\begin{proposition}\label{prop:star_lower3}
Let $\mathcal{C}$ be the class of stars with sparsity $s$ and $s-1$ and suppose there is an estimator $\hat{S}$ satisfying \eqref{eqn:budget} and \eqref{eqn:goal}.  Then
\[
\mu \geq \sqrt{\frac{2(p-s)}{m} \left( \log \frac{s}{2 \varepsilon} + \log \frac{p-s}{p} \right)} \ .
\]
\end{proposition}

\begin{remark}
When $s \ll n$ the bound above scales as $\sqrt{ \frac{\sqrt{n}}{m} \log \frac{s}{\varepsilon}}$.
\end{remark}

Combining the results of Propositions~\ref{prop:star_lower2}~and~\ref{prop:star_lower3} shows that considering $s$-stars the scaling of the signal strength needs to be at least
\[
\max \left\{ \frac{n}{s^2 m} \log \frac{s}{\varepsilon} , \frac{\sqrt{n}}{m} \log \frac{s}{\varepsilon} \right\} \ .
\]
The first term in the maximum above dominates the second when $s^4 \leq n$.  This shows that the performance of Proposition~\ref{prop:star} can only be achieved in that sparsity regime.

\begin{remark}
Note that the setting of the proposition above is slightly different than the one considered in Section~\ref{sec:star}.  However, we present this result here merely to make a remark on the conditions in Proposition~\ref{prop:star} and it only serves an illustrative purpose.  Furthermore the procedure presented in Section~\ref{sec:star} can be easily modified to handle classes considered in the above proposition and have similar performance guarantees to Proposition~\ref{prop:star}.
\end{remark}


\vspace{0.2cm}
\textbf{$s_r ,s_c$-submatrices}
\vspace{0.2cm}

The case of submatrices has been studied in \cite{Balakrishnan_2012}, where the authors consider block-structured activations in matrices.  They provide a lower bound akin to that of Proposition~\ref{prop:int_lower1} and a near optimal procedure.  Our setting is more general as we consider arbitrary sub-matrices of a given dimension.  Nonetheless the same type of lower bound holds in this case as well.

\begin{proposition}\label{prop:submat_lower1}
Let $\mathcal{C}$ be the class of $s_r ,s_c$-submatrices, and for sake of simplicity assume that both $n_r /s_r$ and $n_c /s_c$ are integers.  Suppose there is an estimator satisfying \eqref{eqn:budget} and $\displaystyle{\max_{S \in \mathcal{C}}} \ \E_S \big( d(\hat{S},S) \big) \leq \varepsilon$.  Then
\[
\mu \geq (1-\varepsilon ) \sqrt{ \frac{n}{2 s^2 m}} \ .
\]
\end{proposition}
\begin{proof}
Since both $n_r /s_r$ and $n_c /s_c$ are integers the proof goes the same way as that of Proposition~\ref{prop:int_lower1} by considering any disjoint partition of the original matrix consisting of submatrices of size $s_r \times s_c$.
\end{proof}

However, our procedures do not reach this lower bound, hence the question arises whether the lower bound above is loose or the procedures are suboptimal? We partially answer this question by presenting another simple lower bound with which we illustrate that in certain sparsity regimes the procedure of Proposition~\ref{prop:submat1} is indeed optimal.  Consider the class containing all $s_r \times s_c$ and $s_r \times (s_c -1)$ submatrices, and consider the sub-problem of estimating the support when the active rows are given.  This is a problem of estimating $s_c$ or $s_c -1$ disjoint $s_r$-intervals in a signal of size $s_r \cdot n_c$.  Note that the procedure of Proposition~\ref{prop:submat1} can handle such classes without any modifications.  Now we can directly apply Proposition~\ref{prop:int_lower2} to get the following.

\begin{proposition}
Let $\C$ be the class containing all submatrices of size $s_r \times s_c$ and $s_r \times (s_c -1)$.  Suppose there is an estimator $\hat{S}$ satisfying \eqref{eqn:budget} and \eqref{eqn:goal}.  Then
\[
\mu \geq \sqrt{\frac{2 (n_c - s_c +1)}{s_r m} \left( \log \frac{s}{8 \varepsilon} + \log \frac{n_c - s_c +1}{n_c +1} \right)} \ .
\]
\end{proposition}

When $s_r \approx n_r$ (for instance we have linear sparsity \emph{in the rows}: $s_r = c n_r$ with some $c \in (0,1]$) the performance bound of Proposition~\ref{prop:submat1} becomes essentially identical to the lower bound above.  This shows that in certain regimes that procedure is optimal.  Note that the condition on the number of active rows does not determine the sparsity of the signal, as there is no requirement on the number of active columns for the results to hold.  Also note that by Proposition~\ref{prop:submat2} in certain regimes it is possible to outperform the procedure of Proposition~\ref{prop:submat1} indicating that the gains one can hope for in the case of submatrices depends on the interplay between the dimensions of the problem $n_r ,n_c ,s_r ,s_c$.  On a final note if we assume that the support set is such that either the active rows of active columns (but not necessary both) are consecutive then one can simply modify the procedure presented in Section~\ref{sec:submat} to even reach the lower bound of Proposition~\ref{prop:submat_lower1}.  The exact performance characterization of the case of submatrices with arbitrary dimensions remains an interesting open problem.  



\section{Sample complexity} \label{sec:sample}

In the preceding sections we presented near optimal procedures for structured support recovery using adaptive compressive sensing.  Those procedures provided insight on how to capitalize on the structure of the support sets to achieve performance gains, but paid no regard to the number of measurements that are collected.  However an important aspect of compressive sensing is the possibility to perform estimation using only a small number of observations.  Therefore we now present procedures for structured support recovery that use only a small number of observations.


\subsection{Procedures} \label{sec:sample_proc}

All the procedures presented here are based on an algorithm named Compressive Adaptive Sense and Search (CASS), introduced and analyzed in \cite{ACS_Malloy_2012}.  This procedure is designed to recover non-structured support sets.  To ease presentation we briefly describe and analyze the procedure here, though the reader is referred to \cite{ACS_Malloy_2012} where this has already been done in more detail.


\subsubsection{$s$-sets}

The main idea of the CASS procedure is to use a binary bisection type algorithm to recover the support of the signal.  In a nutshell, CASS begins by partitioning the signal into several bins and deciding if there are any significant components inside each bin.  Then it continues by partitioning the bins deemed to contain signal into new bins and performing the previous step again for those.  By iterating these steps the procedure is able to locate the support in a number of steps that is logarithmic in the dimension of the signal.

Assume the support set is any $s$-sparse set.  Partition $\{ 1,\dots ,n \}$ into $2s$ bins of equal size, denoted by $\vec{I}^{(1)}_1 ,\dots ,\vec{I}^{(1)}_{2s}$.  For each of the $2s$ bins we wish to decide between
\[
H^{(1)}_{i,0} : \ \vec{I}^{(1)}_i \cap S = \emptyset \qquad \textrm{versus} \qquad H^{(1)}_{i,1} : \ \vec{I}^{(1)}_i \cap S \neq \emptyset , \ i=1,\dots ,2s \ .
\]
Once having identified the non-empty bins, we split each of these into two bins of equal size denoted by $\vec{I}^{(2)}_1 ,\dots ,\vec{I}^{(2)}_{2 n_1}$, where $n_1$ denotes the number of bins deemed non empty previously, and do the same as before.  We know that at most $s$ bins can be non-empty, thus we will enforce in our procedure that $n_1 \leq s$.  Hence in step $j$ we consider bins $\vec{I}^{(j)}_1 ,\dots ,\vec{I}^{(j)}_{2 n_{j-1}}$, where $n_{j-1} \leq s$, and test the hypotheses
\[
H^{(j)}_{i,0} : \ \vec{I}^{(j)}_i \cap S = \emptyset \qquad \textrm{versus} \qquad H^{(j)}_{i,1} : \ \vec{I}^{(j)}_i \cap S \neq \emptyset , \ i=1,\dots ,2 n_{j-1} \ .
\]
When $j=\log_2 \frac{n}{2s}$ the bins consist of single components of $\vec{x}$, and the estimator of the support $\hat{S}$ will consist of the ones deemed non-empty in this final step.

To decide between $H^{(j)}_{i,0}$ and $H^{(j)}_{i,1}, \ j=1,\dots ,\log_2 \frac{n}{2s} ; i=1,\dots ,2n_{j-1}$ we collect a single measurement of the form
\[
Y^{(j)}_i = <a \sqrt{j} \1_{\vec{I}^{(j)}_i} , \vec{x}> + W^{(j)}_i , \ j=1,\dots \log_2 \frac{n}{2s};i=1,\dots ,2 n_{j-1} \ ,
\]
where $W^{(j)}_i \sim N(0,1)$ i.i.d., and $a>0$.  The parameter $a>0$ needs to be chosen such that \eqref{eqn:budget} is fulfilled.  Since the length of the bins $I^{(j)}_i$ is $n/ (2^{j} s)$ for every $i=1,\dots ,2 n_{j-1}$, $n_{j-1} \leq s$ and there are $\log_2 \frac{n}{2s}$ steps we can write
\[
\| A \|_F^2 = \sum_{j=1}^{\log_2 \frac{n}{2s}} 2s \frac{n}{2^j s} j a^2 \leq n a^2 \sum_{j=1}^{\infty} j 2^{-(j-1)} = 4 n a^2 \ .
\]
Combining this with \eqref{eqn:budget} yields $a = \sqrt{\frac{m}{4 n}}$.  If the bin $\vec{I}^{(j)}_i$ is non-empty then $\E_S ( Y^{(j)}_i ) \geq \mu \sqrt{\frac{jm}{4 n}}$.  Therefore we conclude that the bin $\vec{I}^{(j)}_i$ is empty if $Y^{(j)}_i \leq \frac{\mu}{2} \sqrt{\frac{jm}{4 n}}$, otherwise we conclude the opposite.  If at any step $j=1,\dots ,\log_2 \frac{n}{2s}$ more than $s$ bins are deemed non-empty, we select those which correspond to the $s$ largest observations.  For the method described above both the type I and type II error probabilities for the test between $H^{(j)}_{i,0}$ and $H^{(j)}_{i,1}$, $j=1,\dots \log_2 \frac{n}{2s};i=1,\dots ,2 n_{j-1}$ can be upper bounded using the Gaussian tail bound
\begin{equation}\label{eqn:tail_bound}
\P (X > \eta) \leq \frac{1}{2} e^{-\eta^2 /2}
\end{equation}
by
\[
\frac{1}{2} e^{- \frac{j m \mu^2}{32 n}} \ .
\]
Hence the probability of error can be bounded from above as follows
\[
\P_S (\hat{S} \neq S) \leq \sum_{j=1}^{\log_2 \frac{n}{2s}} s \ e^{-\frac{jm \mu^2}{32 n}} \ .
\]
Thus whenever $\mu^{2} \geq \frac{32 n}{m} \log \frac{2s}{\varepsilon}$ we have
\[
\P_S (\hat{S} \neq S) \leq \sum_{j=1}^{\log_2 \frac{n}{2s}} s \left( \frac{\varepsilon}{2s} \right)^{j} \leq \sum_{j=1}^{\log_2 \frac{n}{2s}} \left( \frac{\varepsilon}{2} \right)^{j} \leq \varepsilon \ .
\]
When considering the expected Hamming-distance as the error metric we can use the procedure above with probability of error set to $\varepsilon /2s$.  This method then yields an near-optimal estimator for the support recovery problem described in Section~\ref{sec:setting} by collecting at most $2s \log_2 \frac{n}{2s}$ measurements.


\subsubsection{Unions of $s$-intervals}

We can modify the CASS procedure of \cite{ACS_Malloy_2012} to estimate unions of $k$ disjoint $s$-intervals.  Similarly to the procedure presented in Section~\ref{sec:sigstrength_proc} the one discussed here will consist of two phases, a search phase and a refinement phase.  As before, in the search phase we wish to identify the approximate location of the support, that is return a subset of components $\vec{P} \subset \{ 1,\dots ,n\}$ such that $|\vec{P}| \ll n$ and $S \subset \vec{P}$ with high probability.  Again we start by splitting $\{ 1,\dots ,n \}$ into consecutive bins of size $s/2$ denoted by $\vec{P}^{(1)},\dots ,\vec{P}^{(2n/s)}$.  To ease the presentation we assume $2n/s$ is an integer since the case when this is not satisfied can be handled with simple modifications.  The same holds for any divisibility issue that we encounter further on.  Of these bins at least $k$ will consist entirely of signal components.  Roughly speaking we think of these bins as signal components of a vector of size $2n/s$, and use a CASS procedure to find them.  Once that is done, we set $\vec{P}$ as the bins deemed active and their neighboring bins, and move on to the refinement phase.  In the refinement phase we estimate the active components in $\vec{P}$ for instance by using another CASS procedure.

We now describe the method in full detail.  Consider the binning $\vec{P}^{(1)},\dots ,\vec{P}^{(2n/s)}$ described before.  Partition the bins into $4k$ groups denoted by $\vec{I}^{(1)}_1,\dots ,\vec{I}^{(1)}_{4k}$.  For each of these we test the hypothesis
\[
H^{(1)}_{i,0} : \ \vec{I}^{(1)}_i \cap S = \emptyset \qquad \textrm{versus} \qquad H^{(1)}_{i,1} : \ | \vec{I}^{(1)}_i \cap S | \geq s/2 , \ i=1,\dots ,4k \ .
\]
The groups for which $H^{(1)}_{i,1}$ is accepted are split into two in the middle giving us the groups $\vec{I}^{(2)}_1,\dots ,\vec{I}^{(2)}_{2n_1}$.  We now test a similar hypotheses as before for these new groups.  Since at most $3k$ groups can contain signal components, we will specifically enforce $n_1 \leq 3k$.  Iterating this, in step $j$ we have groups denoted by $\vec{I}^{(j)}_1,\dots ,\vec{I}^{(j)}_{2n_{j-1}}$, where $n_{j-1} \leq 3k$, and we wish to decide between
\[
H^{(j)}_{i,0} : \ \vec{I}^{(j)}_i \cap S = \emptyset \qquad \textrm{versus} \qquad H^{(j)}_{i,1} : \ | \vec{I}^{(j)}_i \cap S | \geq s/2 , \ i=1,\dots ,2n_{j-1} \ .
\]
When $j=\log_2 {n/2ks}$ the groups consist of single bins.  The set $\vec{P}$ will consist of the ones for which $H^{(1)}_{i,1}$ is accepted in this final step and the bins adjacent to those.

To decide between $H^{(j)}_{i,0}$ and $H^{(j)}_{i,1}, \ j=1,\dots ,\log_2 \frac{n}{2s} ; i=1,\dots ,2n_{j-1}$ we collect a single measurement of the form
\[
Y^{(j)}_i = <a \sqrt{j} \1_{\vec{I}^{(j)}_i} , \vec{x}> + W^{(j)}_i , \ j=1,\dots \log_2 \frac{n}{2s};i=1,\dots ,2 n_{j-1} \ ,
\]
where $W^{(j)}_i \sim N(0,1)$ i.i.d., and $a>0$.  The parameter $a>0$ needs to be chosen such that \eqref{eqn:budget} is fulfilled.  We will use half of our energy budget for the search phase.  Since the groups $I^{(j)}_i$ contain $n/ (2^{j+1}k)$ components for every $i=1,\dots ,2 n_{j-1}$, $n_{j-1} \leq 3k$ and there are $\log_2 \frac{n}{2ks}$ steps we can write
\[
\| A_{search} \|_F^2 = \sum_{j=1}^{\log_2 \frac{n}{2ks}} 6k \frac{n}{2^{j+1}k} j a^2 = \frac{3}{2} n a^2 \sum_{j=1}^{\log_2 \frac{n}{2s}} j 2^{-(j-1)} = 6 n a^2 \ .
\]
Since we use at most $m/2$ energy in the search phase we get $a = \sqrt{\frac{m}{12 n}}$.  If group $\vec{I}^{(j)}_i$ contains a bin which is contained in $S$, we have $\E_S ( Y^{(j)}_i ) \geq \frac{s \mu}{2} \sqrt{\frac{jm}{12 n}}$.  Therefore we declare that the group contains no signal components if $Y^{(j)}_i \leq \frac{s \mu}{4} \sqrt{\frac{jm}{12 n}}$, otherwise we declare the opposite.  If in step $j=1,\dots ,\log_2 \frac{n}{2ks}$ we accept $H^{(j)}_{i,1}$ for more than $3k$ groups, we choose those corresponding to the highest $3k$ observations.  Considering a single test the type I and type II error probabilities can both be upper bounded using \eqref{eqn:tail_bound} by
\[
\frac{1}{2} e^{- \frac{j s^2 m \mu^2}{384 n}} \ .
\]

It is also possible that neither the null or the alternative is true, and the group contains some bins that intersect with $S$, but are not contained in $S$.  However we need not pay any attention to those, as by construction $\vec{P}$ will also contain neighboring bins of those we deem non-empty.  The probability of either concluding $H^{(j)}_{i,1}$ when the group $\vec{I}^{(j)}_i$ contains no signal or concluding $H^{(j)}_{i,0}$ when in fact $H^{(j)}_{i,1}$ is true can be bounded from above by
\[
\sum_{j=1}^{\log_2 \frac{n}{2ks}} 3k \ e^{-\frac{j s^2 m \mu^2}{384 n}} \ .
\]
Thus whenever $\mu \geq \sqrt{\frac{384 n}{s^2 m} \log \frac{9k}{\varepsilon}}$ we have that
\[
\P_S (S \nsubseteq \vec{P}) \leq \sum_{j=1}^{\log_2 \frac{n}{2ks}} 3k \left( \frac{\varepsilon}{9k} \right)^{j} \leq \sum_{j=1}^{\log_2 \frac{n}{2ks}} \left( \frac{\varepsilon}{3} \right)^{j} \leq \varepsilon /2 \ .
\]

We also have by construction that $|\vec{P}| \leq \tfrac{9}{2}ks$.  Hence in the refinement phase we can measure each component in $\vec{P}$ separately, say, to produce $\hat{S}$.  We have $2m/18ks$ energy for each of the components in $\vec{P}$, hence it is easy to check using \eqref{eqn:tail_bound} that the probability of making an error in the refinement phase is at most
\[
\frac{9ks}{4} e^{- \frac{m \mu^2}{72 ks}} \ .
\]
Whenever $\mu \geq \sqrt{\frac{72 ks}{m} \log \frac{9ks}{2 \varepsilon}}$ the probability above is at most $\varepsilon /2$.  Thus the procedure given an estimator $\hat{S}$ for which $\P_S (\hat{S} \neq S) \leq \varepsilon$ whenever
\[
\mu \geq \sqrt{\max \{ \frac{384 n}{s^2 m} \log \frac{9k}{\varepsilon} , \frac{72 ks}{m} \log \frac{9ks}{2 \varepsilon} \} } \ .
\]

When considering the expected Hamming-distance as the error metric we can use the procedure above with probability of error set to $\varepsilon /2s$ in the search phase and $\varepsilon /2ks$ in the refinement phase.  This method then yields an near-optimal estimator for the support recovery problem described in Section~\ref{sec:setting} by collecting at most $3k \left( \log_2 \frac{n}{2ks} + \tfrac{3}{2} s \right)$ measurements.

\begin{proposition}\label{prop:int_sample}
Consider the class of $k$ disjoint $s$-intervals and suppose $n>ks^3$.  Then the procedure above satisfies \eqref{eqn:budget} and \eqref{eqn:goal} whenever
\[
\mu \geq \sqrt{\frac{768 n}{s^2 m} \log \frac{3\sqrt{2}ks}{\varepsilon}} \ .
\]
Furthermore, the procedure collects at most $3k \left( \log_2 \frac{n}{2ks} + \tfrac{3}{2} s \right)$ observations.
\end{proposition}

\begin{remark}
As with Proposition~\ref{prop:int} the condition on the sparsity is an artifact of the simple method above and can be avoided by using a more elaborate method in the refinement phase, for instance binary search.
\end{remark}


\subsubsection{Unions of $s$-stars}\label{sec:samplecomp_star}

Consider the class of $k$ disjoint $s$-stars.  To ease the discussion we focus on the case $k=1$, but the idea can be applied to larger $k$.  The procedure is very similar to the one used for unions of $s$-intervals, however due to the different nature of the structure we provide a detailed description of the procedure in Appendix~\ref{app:samplecomp_star}.

\vspace{0.1cm}

\begin{proposition}
Consider the class of $s$-stars, and suppose $\sqrt{2n} \geq s^2$.  Then the procedure described in the Appendix satisfies \eqref{eqn:budget} and \eqref{eqn:goal} whenever
\[
\mu \geq \sqrt{\frac{392 n}{s^2 m} \log \frac{9s}{\varepsilon}} \ .
\]
Furthermore, the procedure collects at most $4 \log_2 \frac{p}{4} + 2s \log_2 \frac{p-1}{s} \leq 8 \log_2 n + 2s \log_2 \frac{\sqrt{2n}-1}{s}$ observations.
\end{proposition}

Similar ideas can be used to treat the case of $k$ disjoint $s$-stars when $k>1$, but $k \ll s$.

\vspace{0.1cm}


\subsubsection{$s_r ,s_c$-submatrices}

Consider the class of submatrices of size $s_r \times s_c$ of a matrix of size $n_r \times n_c$, and suppose $s_r \geq s_c$.  The procedure we present now is very similar to the one used for unions of $s$-intervals, hence we only provide an outline and present performance guarantees here.

Once more we break the procedure into two phases, a search phase and a refinement phase.  The aim of the search phase is to find the active columns of the signal matrix, whereas the refinement phase aims to find the active rows once the active columns are found.  If we view the columns of the signal matrix as components of a vector of dimension $n_c$, then finding the active columns can be viewed as estimating an unstructured $s_c$-sparse support set.  Likewise the problem of the refinement phase can be viewed as finding an $s_r$-set in a signal of dimension $n_r$.  Hence we can immediately use the CASS procedure for both sub-problems with modifications similar to those used in the case of unions of $s$-intervals.  Thus we get the following.

\begin{proposition}\label{prop:samplecomp_submat}
Consider the class of $s_r ,s_c$-submatrices and suppose $n_c > s_r^2 /s_c$.  There exists a procedure which yields an estimator satisfying \eqref{eqn:budget} and \eqref{eqn:goal} whenever
\[
\mu \geq \sqrt{\frac{128 n}{s_r^2 m} \log \frac{2s}{\varepsilon}} \ .
\]
Furthermore the estimator takes at most $2s_c \log_2 \frac{n_c}{2 s_c} + 2s_r \log_2 \frac{n_r}{2 s_r}$ measurements.
\end{proposition}
The sketch of the proof of Proposition~\ref{prop:samplecomp_submat} is given in Appendix~\ref{app:samplecomp_submat_proof}.

\begin{remark}
The result above guarantees essentially the same performance as Proposition~\ref{prop:submat1}.  We remark that it is possible to formulate a CASS-type algorithm whose performance would match that in Proposition~\ref{prop:submat2}, by aiming to find only one active column in the first phase.  This requires some modifications to the original CASS procedure which are rather technical.  Hence we did not include the details for the sake of space.
\end{remark}


\subsection{Sample Complexity lower bounds} \label{sec:sample_lower}

Necessary conditions for the sample complexity of compressive sensing have been studied both in the adaptive and non-adaptive setting in \cite{Samplecomp_Saligrama_2013} and \cite{Samplecomp_Saligrama_2014}.  In both works sample complexity was studied for the unstructured case of $s$-sets.  For the non-adaptive setting the authors show in Theorem~4.1 of \cite{Samplecomp_Saligrama_2013} that the sample complexity can be lower bounded by an expression that scales essentially like $s \log \frac{n}{s}$.  Furthermore they also show that the signal to noise ratio plays a role in the sample complexity of compressive sensing, and this phenomenon is also explicitly captured in their bound.  Though the setting considered in their work is slightly different from that in the present work, Theorem~4.1 of \cite{Samplecomp_Saligrama_2013} can be translated into our setting in the following manner.

\begin{lemma}[Theorem~4.1 of \cite{Samplecomp_Saligrama_2013}]\label{lemma:sample_nonadapt_lower}
Consider the class of $s$-sets, and suppose there exists a non-adaptive estimator satisfying \eqref{eqn:budget} and for which $\frac{1}{|\C |} \displaystyle{\sum_{S \in \C}} \ \P_S (\hat{S} \neq S)$ is not asymptotically bounded away from zero as $n,s \to \infty$.  Let $k(n,s)$ denote the number of measurements the estimator makes.  Then
\[
k(n,s) \geq \frac{c s \log \frac{n}{s}}{\log \left( \mu^2 \frac{m}{n} + 1 \right)} \ ,
\]
with some constant $c$.
\end{lemma}

This shows that the procedure presented in the previous section for $s$-sets performs as well in terms of sample complexity as the best non-adaptive procedure.  Furthermore, when estimating structured support sets, potentially less samples are enough to perform accurate estimation.  We now briefly discuss necessary conditions on sample complexity for non-adaptive estimators for the structured classes we examined before.

Consider first the case of unions of $k$ disjoint $s$-intervals.  Without giving a rigorous formal proof we argue that the number of samples required in the non-adaptive case must scale as $k \log \frac{n}{sk}$.  Let $S_1 , \dots ,S_{n/s}$ be consecutive disjoint $s$-intervals of $\{ 1,\dots, n \}$ and let
\[
\C' = \left\{ S \in \C: \ S = \bigcup_{j=1}^{k} S_{i_j}, \ i_1,\dots ,i_k \in \{ 1,\dots ,n/s\} \right\} \ ,
\]
that is unions of intervals that are constructed from  $S_1 , \dots ,S_{n/s}$.  This class roughly behaves like a class of $k$-sparse sets of a vector of dimension $n/s$, except that there is an increase in the relative sensing power arising from the fact that the building blocks of the class are $s$-sets instead of singletons.  This results in that it is possible to detect somewhat weaker signals (see Proposition~\ref{prop:int_nonadapt_lower}), but because of the weak dependence of the sample complexity bound of Lemma~\ref{lemma:sample_nonadapt_lower} on the signal to noise ratio, the scaling of the bound will still be dictated by the numerator.

The class of unions of $k$ disjoint $s$-stars is even more simple to consider.  Suppose $k=1$, and that the center of the star is given by an oracle.  The remaining problem is the estimation of an $s$-sparse set in a vector with dimension roughly $\sqrt{2n}$.  Hence the sample complexity remains essentially the same as that of the unstructured case.

Finally for the class of $s_r ,s_c$-submatrices, if an oracle provides the active columns, the problem reduces to the unions of intervals case.

This shows that the procedures presented in the previous section for structured support recovery perform as well in terms of sample complexity as the best non-adaptive procedures.  It is plausible however that adaptive procedures might outperform non-adaptive ones in terms of sample complexity.  This question was investigated in \cite{Samplecomp_Saligrama_2014}, where the authors provide a necessary condition for any adaptive algorithm to recover unstructured $s$-sets.  The number of samples required is dependent on the signal to noise ratio in this case as well.  Their results show that when the signal to noise ratio is near the boundary where accurate estimation is possible (see Proposition~\ref{prop:vanilla_lower}, and \cite{AS_Rui_2012}) the number of samples needs to scale essentially like $s$.  It is still an open question whether this bound is achievable or not.

Although not yet having a rigorous proof, the authors of this work conjecture that although some performance gain might be present, it is not substantial and the number of samples needs to scale essentially like $s \log \frac{n}{s}$ for adaptive estimators as well, when the signal magnitude is close to the estimation threshold.  The reason behind this conjecture is roughly the following.  Consider the 1-sparse case.  It can be easily seen that by taking one measurement, a fraction of the $n$ hypotheses (namely that the signal component is at coordinate $1,\dots ,n$) remains essentially indistinguishable.  Focusing the next measurement on these potential signal components, again a fraction of them will remain essentially indistinguishable.  With a bit of work this line of reasoning will, in principle, provide a lower bound on the sampling complexity.  However, formalizing this argument is challenging, because each projection does contain some faint amount of information about these ``indistinguishable'' hypotheses.  So one needs to show that these small amounts of information are negligible as a whole, even after collecting multiple projections.  Showing this requires the proof of a sharp information-contraction bound suitable for the adaptive sensing setting.  Nonetheless, the authors conjecture that because of this heuristic, a term that is logarithmic in the dimension should also be present in the sample complexity lower bounds.  In \cite{foucart_rauhut:13} a different compressed sensing setting and framework was considered.  Although this setting is not directly comparable to ours, the authors show that adaptive sensing does not further reduce the sample complexity, which also leads us to believe our conjecture is reasonable.

\vspace{0.1cm}



\section{Final remarks} \label{sec:conc}

In this work we examined the problem of recovering structured support sets through adaptive compressive measurements.  We have seen that by adaptively designing the sensing matrix it is possible to achieve performance gains over non-adaptive protocols, and that the gains can be quite dramatic for instance in the case of $s$-stars.  We have also seen that these gains can be realized by simple and practically feasible estimation procedures.

However, a complete characterization of the problem for the class of submatrices is still missing.  This could prove to be an interesting area for future research considering the practical relevance of that model in gene expression studies.  Furthermore, it remains unclear if the sample complexity of support recovery using compressive measurements can be significantly reduced by adaptively designing the rows of the sensing matrix.  Finally, the procedures of Section~\ref{sec:sample_proc} can be modified using ideas presented in \cite{CSDetection_AriasCastro_2012} to be able to handle signals with arbitrary signs and magnitudes.  Working out the details could prove to be a useful extension to this work.



\appendix

\section{Description of the procedure of Section~\ref{sec:samplecomp_star}}\label{app:samplecomp_star}

We begin with a search phase to find the approximate location of the support.  Again we consider the subsets $\vec{P}^{(i)}, \ i=1,\dots ,p$, where $\vec{P}^{(i)}$ contains all the components whose corresponding edges lie on the vertex $v_i$.  Our goal is to find the center of the $s$-star.  We begin by forming 4 groups $\vec{I}^{(1)}_1 ,\dots ,\vec{I}^{(1)}_4$, where each of them is a union of $p/4$ different $\vec{P}^{(i)}$, and no subset $\vec{P}^{(i)}$ is contained in more than one group.  We then take one measurement per group
\[
Y^{(1)}_i = <a \1_{\vec{I}^{(1)}_i} , \vec{x}> + W^{(1)}_i , \ i=1,\dots ,4,
\]
where $W^{(1)}_i$ i.i.d.  standard normals and $a>0$.  Large measurements should correspond to groups containing a lot of signal components, and particularly the one containing the center of the star.  However, because of the structure of the support and the fact that these groups are not disjoint, large observations may also correspond to groups not containing the center of the star.  Therefore instead of performing hypothesis tests we choose the two highest observations, and consider the groups corresponding to those.  Once we have these groups, we split each in half in the sense that half of the $\vec{P}^{(i)}$ in a given group will form one new group, and the other half will form another new group.  This way we end up with 4 groups, again not disjoint, and do the same as before.  Let the groups in step $j$ be denoted by $\vec{I}^{(j)}_1 ,\dots ,\vec{I}^{(j)}_4$.  The measurements we collect are
\[
Y^{(j)}_i = <a \sqrt{j} \1_{\vec{I}^{(j)}_i} , \vec{x}> + W^{(j)}_i , \ j=1,\dots ,\log_2 \frac{p}{4}; i=1,\dots ,4 \ .
\]
In the final step $j= \log_2 \frac{p}{4}$ each group consists of a single $\vec{P}^{(i)}$.  The output set of the search phase $\vec{P}$ will consist of the union of those two groups for which the final observation is largest.

First we specify the parameter $a$ so as to ensure we don't use more than half of our measurement budget.  Each $\vec{I}^{(j)}_i$ contains at most $(p-1) \frac{p}{2^{j+1}} = n /2^j$ components $i=1,\dots ,4$, and $j=1,\dots ,\log_2 \frac{p}{4}$, hence
\[
\| A_{search} \|_F^2 \leq \sum_{j=1}^{\log_2 \frac{p}{4}} \frac{n}{2^{j-2}} j a^2 \leq 8 n a^2 \ .
\]
Therefore $a =\sqrt{\frac{m}{16 n}}$ ensures we use at most $m/2$ energy in the search phase.

Now we need to show that $S \subset \vec{P}$ with high probability.  Without loss of generality suppose that $\vec{I}^{(j)}_1 ,\dots ,\vec{I}^{(j)}_4$ are indexed such that the center of the star is in group $\vec{I}^{(j)}_1$, and for the number of signal components in $\vec{I}^{(j)}_i$ denoted by $N^{(j)}_i$ we have $N^{(j)}_i \geq N^{(j)}_{i+1}$.  Hence $\vec{I}^{(j)}_1$ contains exactly $s$ components, and because $\displaystyle{\sum_{i=2}^{4}} \ N^{(j)}_i \leq s$ we know $N^{(j)}_3 \leq s/2$.  Using this we conclude that in each step $j=1,\dots ,\log_2 \frac{p}{4}$ the probability that $Y^{(j)}_1 < \max \{ Y^{(j)}_3 , Y^{(j)}_4 \}$ can be bounded from above with \eqref{eqn:tail_bound} by
\[
3 \cdot \frac{1}{2} e^{- \frac{j s^2 m \mu^2}{392 n}} \ .
\]
From this we get that whenever $\mu \geq \sqrt{\frac{392 n}{s^2 m} \log \frac{9}{2 \varepsilon}}$ we have
\[
\P_S (S \nsubseteq \vec{P}) \leq \sum_{j=1}^{\log_2 \frac{p}{4}} \left( \frac{\varepsilon}{3} \right)^j \leq \varepsilon /2 \ .
\]
By construction we make $4 \log_2 \frac{p}{4}$ observations in this phase, and also $|\vec{P}| \leq 2 (p-1)$.

In the search phase we can directly apply the CASS procedure on $\vec{P}$ to estimate the support.  Since $\sqrt{2n} > p-1$ we know that whenever $\mu \geq \sqrt{\frac{64 \sqrt{2n}}{m} \log \frac{4s}{\varepsilon}}$ the probability of error is at most $\varepsilon /2$, and we take at most $2s \log_2 \frac{p-1}{s}$ measurements.  When considering $\E_S (|\hat{S} \triangle S|)$ as the error metric one can set the probability of error to $\varepsilon /2s$ and use the procedure above.

\section{Sketch proof of Proposition~\ref{prop:samplecomp_submat}}\label{app:samplecomp_submat_proof}

We use half the energy for the search phase, and half for the refinement phase.  In step $j$ of the search phase the groups $\vec{I}^{(j)}$ contain $n/2^j s_c$ components and there are at most $2s_c$ components.  Hence the energy used is at most
\[
\sum_{j=1}^{\log_2 \frac{n_c}{2s_c}} 2s_c \frac{n}{2^j s_c} j a^2 = 4 n a^2 \ .
\]
Thus $a = \sqrt{\frac{m}{8 n}}$.  This means that for the probability of error we have
\[
\sum_{j=1}^{\log_2 \frac{n_c}{2s_c}} 2s_c \frac{1}{2} e^{- \frac{s_r^2 j m \mu^2}{64 n}} \ ,
\]
so whenever $\mu \geq \sqrt{\frac{64 n}{s_r^2 m} \log \frac{2 s_c}{\varepsilon}}$ the probability of error is at most $\varepsilon /2$.

In the refinement phase the energy used is
\[
\sum_{j=1}^{\log_2 \frac{n_r}{2s_r}} 2s_r \frac{n_r s_c}{2^j s_r} j a^2 = 4 n_r s_c a^2 \ ,
\]
hence $a = \sqrt{\frac{m}{8 n_r s_c}}$.  Therefore the probability of error is at most
\[
\sum_{j=1}^{\log_2 \frac{n_r}{2s_r}} 2s_r \frac{1}{2} e^{- \frac{s_c j m \mu^2}{64 n_r}} \ ,
\]
which means whenever $\mu \geq \sqrt{\frac{64 n_r}{s_c m} \log \frac{2 s_r}{\varepsilon}}$ the probability of error is at most $\varepsilon /2$.

Considering the expected Hamming-distance as the error metric, we can use the procedure above with probability of error set to $\varepsilon /2s$.

\section{Removing the expectation from the energy constraint \eqref{eqn:budget}}\label{app:no_expectations}

We now investigate what difference would it make if we considered a more demanding energy constraint by removing the expectation from \eqref{eqn:budget}.  That is, we now wish to consider algorithms that satisfy
\[
\sup_{S \in \C} \ \| A \|^2_F \ =\ \E_S\left(\sum_j \|A_j\|^2 \right)\ \leq m\ .
\]

First, note that all the lower bounds remain valid with the latter constraint as well, since the constraint $\E (\| A \|_F^2 ) \leq m$ is more forgiving then $\| A \|_F^2 \leq m$.

Considering the procedures in this paper, we begin by noting that the CASS procedure and thus all procedures in Section~4.1 that are derived from it already satisfy $\| A \|_F^2 \leq m$.  This comes at a price of an increase in the constants for the $s$-sets and possibly for other classes as well, though we can only make a rigorous claim for the unstructured case as we don't put an effort into finding the correct constants for structured classes.  All that is left is to address the procedures in Section~3.1.  Since the basis of these procedures is the SLRT described at the beginning of Section~3.1 used for the $s$-sets (Proposition~\ref{prop:vanilla1}), we will only discuss this in detail as results for all other procedures follow similarly.

As a reminder, the procedure for $s$-sets consists of independently performing a Sequential Likelihood-ratio test (SLRT) for each component to assess whether that component is zero or not.  To carry out the test for $\vec{x}_i$ we take
\[
N_i = \inf \left\{ n \in \N:\ \sum_{j=1}^n \log \frac{\d \P_0 (Y_{i,j})}{\d \P_1 (Y_{i,j})} \notin (l,u) \right\}
\]
measurements, where $\log \tfrac{\beta}{1-\alpha} =l <0< u=\log \tfrac{1-\beta}{\alpha}$ are the lower and upper stopping boundaries.  Lemma~1 establishes an upper bound on the expectation of $N_i$ under the null and alternative respectively.  Suppose $H_0$ is true.  In this case the upper bound on the expected energy used by the test is $t_0:= \tfrac{2}{\mu^2} \log \tfrac{2s}{\varepsilon}$ since we set $\beta = \tfrac{\varepsilon}{2s}$.  For our purposes, we need more than a bound on the expectation of $N_i$, and it is enough to show
\begin{equation}\label{eqn:x}
\P \left( \sum_{i\notin S} N_i > c(n-s)t_0 \right) \leq c'\varepsilon \ ,
\end{equation}
with some universal constants $c,c'$.  If this (and a comparable result for $i\in S$) were true, then a union bound would give that the probability that the procedure uses more then $c m$ energy is at most $2c' \varepsilon$.  One then could construct a similar procedure as before with the exception that it is forced to stop once the precision budget is exhausted.  By the previous result this happens with probability proportional to $\varepsilon$.  Hence the minimum signal strength required by a procedure satisfying $\| A \|_F^2 \leq m$ for support recovery would still be on the same order as before, only the constants would need to be adjusted.

To show the result above we need a concentration inequality.  As a start, we show a simple tail bound for $N_i$ under the null.
\begin{align*}
\P_0 (N_i >ct_0) & \leq \P_0 \left( \sum_{j=1}^{ct_0} \log \frac{\d \P_0 (Y_{i,j})}{\d \P_1 (Y_{i,j})} > l \right) \\
& \leq \P_0 \left( \sum_{j=1}^{ct_0} \log \frac{\d \P_0 (Y_{i,j})}{\d \P_1 (Y_{i,j})} > \log \beta \right) \\
& \leq \frac{1}{2} \left( \frac{\varepsilon}{2s} \right)^{\tfrac{(c-1)^2}{2c}} \ .
\end{align*}
when $c>2$.  We continue by using the Craig-Bernstein inequality \cite{Craig_1933} that states that whenever the independent random variables $U_1,\dots ,U_n$ satisfy the moment condition
\[
\E \left( |U_i - \E (U_i)|^k \right) \leq \frac{\Var (U_i)}{2} k! h^{k-2},\ i=1,\dots ,n\ ,
\]
with some $h>0$ then we have
\[
\P \left( \frac{1}{n} \sum_{i=1}^n (U_i - \E (U_i)) \leq \frac{z}{n\delta} + \frac{n\delta \Var (\tfrac{1}{n} \sum_{i=1}^n U_i)}{2(1-c)} \right) \leq e^{-z} \ ,
\]
for $0<h\delta \leq c<1$ and $z>0$.  We thus need to refine the calculations above to get a general moment bound for $N_i$ and then we will use the inequality above with $c=1/2, \delta = 1/2h$ and an appropriate $z$.  We start with the moment condition.
\begin{align*}
\E (N_i^k) & = \sum_{j=1}^\infty j^k \P (N_i = j) \\
& \leq \sum_{c=1}^\infty (ct_0)^k \P ((c-1)t_0 < N_i \leq ct_0) \\
& \leq \sum_{c=1}^\infty (ct_0)^k \P ((c-1)t_0 < N_i) \\
& \leq t_0^k \left( 2^k + \frac{1}{\varepsilon} \sum_{c=3}^\infty c^k \varepsilon^{c/2} \right) \ ,
\end{align*}
using the tail bound on $N_i$ (also using $\varepsilon \leq 1/2$).  We upper bound the sum in the last expression by an integral.
\begin{align*}
\sum_{c=3}^\infty c^k 2^{-c/2} & \leq \int_0^\infty (x+1)^k \sqrt{\varepsilon}^x \d x \\
& = \left[ \frac{2(x+1)^k \sqrt{\varepsilon}^x}{\log \varepsilon} \right]_0^\infty - \frac{2k}{\log \varepsilon} \int_0^\infty (x+1)^{k-1} \sqrt{\varepsilon}^x \d x \\
& = \frac{2}{\log \tfrac{1}{\varepsilon}} + \frac{2k}{\log \tfrac{1}{\varepsilon}} \int_0^\infty (x+1)^{k-1} \sqrt{\varepsilon}^x \d x \\
& = \dots \\
& = \sum_{l=0}^k \left( \frac{2}{\log \tfrac{1}{\varepsilon}} \right)^{l+1} \frac{k!}{(k-l)!} \\
& \leq k! \sum_{l=0}^\infty \left( \frac{2}{\log \tfrac{1}{\varepsilon}} \right)^l \leq k! \frac{\log \tfrac{1}{\varepsilon}}{\log \tfrac{1}{\varepsilon} -2} \ .
\end{align*}
Plugging this back yields
\[
\E (N_i^k) \leq t_0^k (2^k + \frac{1}{\varepsilon} k! \frac{\log \tfrac{1}{\varepsilon}}{\log \tfrac{1}{\varepsilon} -2}) \leq k! (K t_0)^k \ .
\]
Since the variance is on the order of $t_0^2$ (say let $\Var (N_i) = K" t_0^2$) this shows that the moment condition above is satisfied with $h= K' t_0$, where $K'$ is some constant.  Hence taking $z=\log \tfrac{1}{\varepsilon}$ the Craig-Bernstein inequality yields
\[
\P \left( \sum_{i\notin S} (N_i - \E (N_i)) \leq 2K' \log \frac{1}{\varepsilon} t_0 + \frac{K"}{2 K'}(n-s)t_0 \right) \leq \varepsilon \ .
\]
Unless $\varepsilon$ is very small (less than $e^{-(n-s)}$), the expression on the left side of the inequality above is upper bounded by $c(n-s)t_0$ and thus we have shown \eqref{eqn:x}.


\bibliographystyle{acm}
\bibliography{CS_references}

\end{document}